\documentclass[12pt]{amsart}
\usepackage{color}
\usepackage{amsmath, amssymb, amsthm, graphics, colordvi}
\usepackage[margin=1.1in]{geometry}
\usepackage[english]{babel}
\usepackage{cleveref}
\usepackage{multicol}
\usepackage{algorithm}
\usepackage{algorithmic}
\newtheorem{defn0}{Definition}[section]
\newtheorem{prop0}[defn0]{Proposition}
\newtheorem{thm0}[defn0]{Theorem}
\newtheorem{lemma0}[defn0]{Lemma}
\newtheorem{corollary0}[defn0]{Corollary}
\newtheorem{example0}[defn0]{Example}
\newtheorem{remark0}[defn0]{Remark}
\newtheorem{conjecture0}[defn0]{Conjecture}

\newenvironment{definition}{\medskip \begin{defn0}}{\end{defn0}}
\newenvironment{proposition}{\medskip \begin{prop0}}{\end{prop0}}
\newenvironment{theorem}{\medskip \begin{thm0}}{\end{thm0}}
\newenvironment{lemma}{\medskip \begin{lemma0}}{\end{lemma0}}
\newenvironment{corollary}{\medskip \begin{corollary0}}{\end{corollary0}}
\newenvironment{example}{\medskip \begin{example0}\rm}{\end{example0}}
\newenvironment{remark}{ \medskip\begin{remark0}\rm}{\end{remark0}}

\newcount\minleft
\newcount\timehour
\def\thetime{\timehour=\time
\divide\timehour by60 \minleft=\timehour \multiply\minleft by -60
\advance\minleft by\time \ifnum\time>720\advance\timehour
by-12\fi\relax
\number\timehour:\ifnum\minleft<10 %
    0\fi\relax\number\minleft
    \ifnum\time>720~pm \else~am\fi}

\newcommand{\m}{\mathfrak m}
\def\res{{\bf k}}                   %% residual field

\def\ann{\operatorname{Ann_R}}
\def\soc{\operatorname{soc}}
\def\socdeg{\operatorname{socdeg}}

\def\embd{\operatorname{emb dim}}
\def\gcl{\operatorname{gcl}}
\def\rad{\operatorname{rad}}
\def\rk{\operatorname{rk}}

\def\HF{\operatorname{HF}}

\begin{document}
\title[Computing minimal Gorenstein covers]{{\bf
Computing minimal Gorenstein covers}}
%\author[J. Elias]{J. Elias ${}^{*}$}
\author[J. Elias]{Juan Elias ${}^{*}$}
\thanks{${}^{*}$
Partially supported by  MTM2016-78881-P}
\address{Juan Elias
\newline \indent Departament de Matem\`atiques i Inform\`atica
\newline \indent Facultat de Matem\`{a}tiques
\newline \indent Universitat de Barcelona
\newline \indent Gran Via 585, 08007
Barcelona, Spain}  \email{{\tt elias@ub.edu}}

%\author[R. Homs]{R. Homs ${}^{**}$}
\author[R. Homs]{Roser Homs ${}^{**}$}
\thanks{${}^{**}$
Partially supported by MTM2016-78881-P, BES-2014-069364 and EEBB-I-17-12700.}

\address{Roser Homs
\newline \indent Departament de Matem\`atiques i Inform\`atica
\newline \indent Facultat de Matem\`{a}tiques
\newline \indent Universitat de Barcelona
\newline \indent Gran Via 585, 08007
Barcelona, Spain}  \email{{\tt rhomspon7@ub.edu}}

%\author[B. Mourrain]{B. Mourrain ${}^{***}$}
\author[B. Mourrain]{Bernard Mourrain ${}^{***}$}
\thanks{${}^{***}$
Partially supported by the European Union's Horizon 2020 research and innovation programme under the Marie Skłodowska-Curie grant agreement No 813211 of POEMA”.
}

\address{Bernard Mourrain
\newline \indent AROMATH
\newline \indent INRIA - Sophia Antipolis M\'editeran\'ee
\newline \indent 2004 route des Lucioles, 06902 Sophia Antipolis, France
} \email{{\tt bernard.mourrain@inria.fr}}

\begin{abstract}
We analyze and present an effective solution to the minimal Gorenstein cover problem:
given a local Artin $\res$-algebra $A=\res[\![x_1,\dots,x_n]\!]/I$, compute an Artin Gorenstein $\res$-algebra $G=\res[\![x_1,\dots,x_n]\!]/J$ such that $\ell(G)-\ell(A)$ is minimal. We approach the problem by using Macaulay's inverse systems and a modification of the integration method for inverse systems to compute Gorenstein covers. We propose new characterizations of the minimal Gorenstein cover and present a new algorithm for the effective computation of the variety of all minimal Gorenstein covers of $A$ for low Gorenstein colength. Experimentation illustrates the practical behavior of the method.
\end{abstract}

\maketitle

\section{Introduction}

Given a local Artin $\res$-algebra $A=R/I$, with $R=\res[\![x_1,\dots,x_n]\!]$, an interesting problem is to find how far is it from being Gorenstein. In \cite{Ana08}, Ananthnarayan introduces  for the first time the notion of Gorenstein colength, denoted by $\gcl(A)$, as the minimum of $\ell(G)-\ell(A)$ among all Gorenstein Artin $\res$-algebras $G=R/J$ mapping onto $A$. Two natural questions arise immediately:

\medskip

{\sc Question A}: How we can explicitly compute the Gorenstein colength of a given local Artin $\res$-algebra $A$?

\medskip

{\sc Question B}: Which are its minimal Gorenstein covers, that is, all Gorenstein rings $G$ reaching the minimum $\gcl(A)=\ell(G)-\ell(A)$?

\medskip
Ananthnarayan generalizes some results by Teter \cite{Tet74} and Huneke-Vraciu \cite{HV06} and provides a characterization of rings of $\gcl(A)\leq 2$ in terms of the existence of certain self-dual ideals $\mathfrak{q}\in A$ with respect to the canonical module $\omega_A$ of $A$ satisfying $\ell(A/\mathfrak{q})\leq 2$. For more information on this, see \cite{Ana08} or \cite[Section 4]{EH18}, for a reinterpretation in terms of inverse systems. Later on, Elias and Silva (\cite{ES17}) address the problem of the colength from the perspective of Macaulay's inverse systems. In this setting, the goal is to find polynomials $F\in S$ such that $I^\perp\subset \langle F\rangle$ and $\ell(\langle F\rangle)-\ell(I^\perp)$ is minimal. Then the Gorenstein $\res$-algebra $G=R/\ann F$ is a minimal Gorenstein cover of $A$. A precise characterization of such polynomials $F\in S$ is provided for $\gcl(A)=1$ in \cite{ES17} and for $\gcl(A)=2$ in \cite{EH18}.

However, the explicit computation of the Gorenstein colength of a given ring $A$ is not an easy task even for low colength - meaning $\gcl(A)$ equal or less than 2 - in the general case. For examples of computation of colength of certain families of rings, see \cite{Ana09} and \cite{EH18}.

On the other hand, if $\gcl(A)=1$, the Teter variety introduced in \cite[Proposition 4.2]{ES17} is precisely the variety of all minimal Gorenstein covers of $A$ and \cite[Proposition 4.5]{ES17} already suggests that a method to compute such covers is possible.

\bigskip
In this paper we address questions A and B by extending the previous definition of Teter variety of a ring of Gorenstein colength 1 to the variety of minimal Gorenstein covers $MGC(A)$ where $A$ has arbitrary Gorenstein colength $t$. We use a constructive approach based on the integration method to compute inverse systems proposed by Mourrain in \cite{Mou96}.

In section 2 we recall the basic definitions of inverse systems and introduce the notion of integral of an $R$-module $M$ of $S$ with respect to an ideal $K$ of $R$, denoted by $\int_K M$. Section 3 links generators $F\in S$ of inverse systems $J^\perp$ of Gorenstein covers $G=R/J$ of $A=R/I$ with elements in the integral $\int_{\m^t} I^\perp$, where $\m$ is the maximal ideal of $R$ and $t=\gcl(A)$. This relation is described in \Cref{F} and \Cref{propint} sets the theoretical background to compute a $\res$-basis of the integral of a module extending Mourrain's integration method.

In section 4, \Cref{ThMGC} proves the existence of a quasi-projective sub-variety $MGC^n(A)$ whose set of closed points are associated to polynomials $F\in S$ such that $G=R/\ann F$ is a minimal Gorenstein cover of $A$. Section 5 is devoted to algorithms: explicit methods to compute a $\res$-basis of $\int_{\m^t}I^\perp$ and $MGC(A)$ for colengths 1 and 2. Finally, in section 6 we provide several examples of the minimal Gorenstein covers variety and list the comptutation times of $MGC(A)$ for all analytic types of $\res$-algebras with $\gcl(A)\leq 2$ appearing in Poonen's classification in \cite{Poo08a}.

All algorithms appearing in this paper have been implemented in \emph{Singular}, \cite{DGPS}, and the library \cite{E-InvSyst14} for inverse system has also been used.

\medskip
{\sc Acknowledgements:} The second author wants to thank the third author for the opportunity to stay at INRIA Sophia Antipolis - M\'editerran\'ee (France) and his hospitality during her visit on the fall of 2017, where part of this project was carried out. This stay was financed by the Spanish Ministry of Economy and Competitiveness through the Estancias Breves programme (EEBB-I-17-12700).

\section{Integrals and inverse systems}

Let us consider the regular local ring $R=\res[\![x_1,\dots,x_n]\!]$ over an arbitrary field $\res$, with maximal ideal $\m$. Let $S=\res[y_1,\dots,y_n]$ be the polynomial ring over the same field $\res$. Given $\alpha=(\alpha_1,\dots,\alpha_n)$ in $\mathbb{N}^n$, we denote by $x^\alpha$ the monomial $x_1^{\alpha_1}\cdots x_n^{\alpha_n}$ and set $\vert\alpha\vert=\alpha_1+\dots+\alpha_n$. Recall that $S$ can be given an $R$-module structure by contraction:
$$
\begin{array}{cccc}
  R\times S & \longrightarrow & S &   \\
  (x^\alpha, y^\beta) & \mapsto & x^\alpha \circ y^\beta = &
  \left\{
    \begin{array}{ll}
      y^{\beta-\alpha}, &  \beta \ge \alpha; \\
      0, & \mbox{otherwise.}
    \end{array}
  \right.
\end{array}
$$

The Macaulay inverse system of $A=R/I$ is the sub-$R$-module $I^\perp=\lbrace G\in S\mid I\circ G=0\rbrace$ of $S$. This provides the order-reversing bijection between $\m$-primary ideals $I$ of $R$ and finitely generated sub-$R$-modules $M$ of $S$ described in Macaulay's duality. As for the reverse correspondence, given a sub-$R$-module $M$ of $S$, the module $M^\perp$ is the ideal $\ann M=\lbrace f\in R\mid f\circ G=0\,\mbox{ for any } G\in M\rbrace$ of $R$. Moreover, it characterizes zero-dimensional Gorenstein rings $G=R/J$ as those with cyclic inverse system $J^\perp=\langle F\rangle$, where $\langle F\rangle$ is the $\res$-vector space $\langle x^\alpha\circ F:\vert\alpha\vert\leq \deg F\rangle_\res$. For more details on this construction, see \cite{ES17} and \cite{EH18}.

Consider an Artin local ring $A=R/I$ of socle degree $s$ and inverse system $I^\perp$. We are interested in finding Artin local rings $R/\ann F$ that cover $R/I$, that is $I^\perp\subset \langle F\rangle$, but we also want to control how apart are those two inverse systems. In other words, given an ideal $K$, we want to find a Gorenstein cover $\langle F\rangle$ such that $K\circ \langle F\rangle\subset I^\perp$. Therefore it makes sense to think of an inverse operation to contraction.

\begin{definition}[Integral of a module with respect to an ideal] Consider an $R$-submodule $M$ of $S$. We define the integral of $M$ with respect to the ideal $K$, denoted by $\int_K M$, as $$\int_K M=\lbrace G\in S\mid K\circ G\subset M\rbrace.$$

\end{definition}

Note that the set $N=\lbrace G\in S\mid K\circ G\subset M\rbrace$ is, in fact, an $R$-submodule $N$ of $S$ endowed with the contraction structure. Indeed,
given $G_1,G_2\in N$ then $K\circ (G_1+G_2)=K\circ G_1+K\circ G_2\subset M$, hence $G_1+G_2\in N$.
For all $a\in R$ and  $G\in N$ we have  $K\circ (a\circ G)=aK\circ G=a\circ (K\circ G)\subset M$, hence $a\circ G\in N$.

\begin{proposition}
\label{integral}
Let $K$ be an $\m$-primary ideal of $R$ and let $M$ be a finitely generated sub-$R$-module of $S$. Then $$\int_K M=\left(KM^\perp\right)^\perp.$$
\end{proposition}
\begin{proof}
Let $G\in\left(KM^\perp\right)^\perp$. Then $\left(KM^\perp\right)\circ G=0$, so $M^\perp\circ\left(K\circ G\right)=0$. Hence $K\circ G\subset M$, i.e. $G\in\int_K M$. We have proved that $\left(KM^\perp\right)^\perp\subseteq\int_K M$.
Now let $G\in\int_K M$. By definition, $K\circ G\subset M$, so $M^\perp\circ\left(K\circ G\right)=0$ and hence $\left(M^\perp K\right)\circ G=0$. Therefore, $G\in\left(M^\perp K\right)^\perp$.
\end{proof}

One of the key results of this paper is the effective computation of  $\int_K M$ (see \Cref{AlINT}).
Last result gives a method for the computation of this module by computing two Macaulay duals. However, since computing Macaulay duals is expensive, \Cref{AlINT} avoids the computation of such duals.

\begin{remark}\label{incl} The following properties hold:
\noindent
$(i)$ Given $K\subset L$ ideals of $R$ and $M$ $R$-module, if $K\subset L$, then $\int_L M\subset\int_K M.$
\noindent
$(ii)$
Given $K$ ideal of $R$ and $M\subset N$ $R$-modules, if $M\subset N$, then $\int_K M\subset\int_K N.$
\noindent
$(iii)$
Given any $R$-module $M$, $\int_ R M=M$.
\end{remark}

The inclusion $K\circ \int_K M\subset M$ follows directly from the definition of integral. However, the equality does not hold:

\begin{example}\label{Ex1}
Let us consider $R=\res[\![x_1,x_2,x_3]\!]$, $K=(x_1,x_2,x_3)$, $S=\res[y_1,y_2,y_3]$ and $M=\langle y_1y_2,y_3^3\rangle$. 
We can compute Macaulay duals with the \emph{Singular} library \textbf{Inverse-syst.lib}, see \cite{E-InvSyst14}. We get $\int_K M=\langle y_1^2,y_1y_2,y_1y_3,y_2^2,y_2y_3,y_3^4 \rangle$ by \Cref{integral} and hence $K\circ \int_K M=\langle y_1,y_2,y_3^3\rangle\subsetneq M.$
\end{example}

We also have the inclusion $M\subset\int_K K\circ M$. Indeed, for any $F\in M$, $K\circ F\subset K\circ M$ and hence $F\in\int_K K\circ M=\lbrace G\in S\mid K\circ G\subset K\circ M\rbrace$. Again, the equality does not hold.

\begin{example} Using the same example as in \Cref{Ex1}, we get
$K\circ M=\m\circ \langle y_1y_2,y_3^3\rangle=\langle y_1,y_2,y_3^2\rangle,$ and
$\int_K( K\circ M)=\left(K (K\circ M)^\perp\right)^\perp=\langle y_1^2,y_1y_2,y_1y_3,y_2^2,y_2y_3,y_3^2\rangle\nsubseteq M.$
\end{example}

\begin{remark}
Note that if we integrate with respect to a principal ideal $K=(f)$ of $R$, then 
$\int_K M=\lbrace G\in S\mid f\circ G\in M\rbrace$. Hence in this case we will denote it by $\int_f M$.
\end{remark}

In particular, if we consider a principal monomial ideal $K=(x^\alpha)$, then the expected equality for integrals $$x^{\alpha}\circ\int_{x^{\alpha}}M=M$$ holds. Indeed, for any $m\in M$, take $G=y^{\alpha}m$. Since $x^{\alpha}\circ y^{\alpha}=1$, then $x^\alpha\circ y^{\alpha}m=m$ and the equality is reached.

\begin{remark}
In general, $\int_{x^{\alpha}} x^{\alpha}\circ M\neq x^{\alpha}\circ\int_{x^{\alpha}} M$, hence the inclusion $M\subset\int_K K\circ M$ is not an equality even for principal monomial ideals. See \Cref{r1}.
\end{remark}

Let us now consider an even more particular case: the integral of a cyclic module $M=\langle F\rangle$ with respect to the variable $x_i$. Since the equality $x_i\circ \int_{x_i}M=M$ holds, there exists $G\in S$ such that $x_i\circ G=F$. This polynomial $G$ is not unique because it can have any constant term with respect to $x_i$, that is $G=y_iF+p(y_1,\dots,\hat{y}_i,\dots,y_n)$. However, if we restrict to the non-constant polynomial we can define the following:

\begin{definition}[$i$-primitive]
The $i$-primitive of a polynomial $f\in S$ is the polynomial $g\in S$, denoted by $\int_i f$, such that
\begin{enumerate}
\item[(i)] $x_i\circ g=f$,
\item[(ii)] $g\vert_{y_i=0}=0$.
\end{enumerate}
\end{definition}

In \cite{EM07}, Elkadi and Mourrain proposed a definition of $i$-primitive of a polynomial in a zero-characteristic setting using the derivation structure instead of contraction. Therefore, we can think of the integral of a module with respect to an ideal as a generalization of their $i$-primitive.

Since we are considering the $R$-module structure given by contraction, the $i$-primitive is precisely
$$\int_i f=y_if.$$

Indeed, $x_i\circ(y_if)=f$ and $(y_if)\mid_{y_i=0}=0$, hence (i) and (ii) hold.
Uniqueness can be easily proved. Consider $g_1,g_2$ to be $i$-primitives of $f$. Then $x_i\circ (g_1-g_2)=0$ and hence $g_1-g_2=p(y_1,\dots,\hat{y}_i,\dots,y_n)$. Clearly $(g_1-g_2)\vert_{y_i=0}=p(y_1,\dots,\hat{y}_i,\dots,y_n)$. On the other hand, $(g_1-g_2)\vert_{y_i=0}=g_1\vert_{y_i=0}-g_2\vert_{y_i=0}=0$. Hence $p=0$ and $g_1=g_2$.

\begin{remark}\label{r1} Note that, by definition, $x_k\circ \int_k f=f$. Any $f$ can be decomposed in $f=f_1+f_2$, where the first term is a multiple of $y_k$ and the second has no appearances of this variable. Then
$$\int_k x_k\circ f=\int_k x_k\circ f_1+\int_k x_k\circ f_2=f_1+\int_k 0=f_1.$$
\noindent
Therefore, in general, $$f_1=\int_k x_k\circ f\neq x_k\circ \int_k f=f.$$
However, for all $l\neq k$, $$\int_l x_k\circ f=\frac{y_lf_1}{y_k}=x_k\circ \int_l f.$$
\end{remark}

\bigskip

Let us now recall Theorem 7.36 of Elkadi-Mourrain in \cite{EM07}, which describes the elements of the inverse system $I^\perp$ up to a certain degree $d$. We define $\mathcal{D}_d=I^\perp\cap S_{\leq d}$, for any $1\leq d\leq s$, where $s=\socdeg(A)$. Since $\mathcal{D}_s=I^\perp$, this result leads to an algorithm proposed by the same author to obtain a $\res$-basis of an inverse system. We rewrite the theorem using the contraction setting instead of derivation.

\begin{theorem}[Elkadi-Mourrain]\label{EM}
Given an ideal $I=(f_1,\dots,f_m)$ and $d>1$. Let $\lbrace b_1,\dots,b_{t_{d-1}}\rbrace$ be a $\res$-basis of $\mathcal{D}_{d-1}$. The polynomials of $\mathcal{D}_d$ with no constant term, i.e. no terms of degree zero, are of the form
\begin{equation}\label{thm}
\Lambda=\sum_{j=1}^{t_{d-1}}\lambda_j^1\int_1 b_j\vert_{y_2=\cdots=y_n=0}+\sum_{j=1}^{t_{d-1}}\lambda_j^2\int_2 b_j\vert_{y_3=\cdots=y_n=0}+\dots+\sum_{j=1}^{t_{d-1}}\lambda_j^n\int_n b_j,\quad\lambda_j^k\in\res,
\end{equation}
such that
\begin{equation}\label{cond1}
\sum_{j=1}^{t_{d-1}}\lambda_j^k (x_l\circ b_j)-\sum_{j=1}^{t_{d-1}}\lambda_j^l(x_k\circ b_j)=0, 1\leq k<l\leq n,
\end{equation}
and
\begin{equation}\label{cond2}
\left(f_i\circ\Lambda\right)(0)=0, \mbox{ for } 1\leq i\leq m.
\end{equation}
\end{theorem}

See \cite{Mou96} or \cite{EM07} for a proof.

\section{Using integrals to obtain Gorenstein covers of Artin rings}

Let us start by recalling the definitions of Gorenstein cover and Gorenstein colength of a local equicharacteristic Artin ring $A=R/I$ from \cite{EH18}:

\begin{definition}
We say that $G=R/J$, with $J=\ann F$, is a Gorenstein cover of $A$ if and only if $I^\perp\subset\langle F\rangle$.
The Gorenstein colength of $A$ is
  $$
  \gcl(A)=\min\{\ell(G)-\ell(A)\mid G \text{ is a Gorenstein cover of  } A\}.
  $$
$A$ Gorenstein cover $G$ of an Artin ring $A$ is minimal
  if $\ell(G)=\ell(A)+\gcl(A)$.
\end{definition}

For all $F\in S$ defining a Gorenstein cover of $A$
  we consider the colon ideal $K_F$ of $R$ defined by
  $$
  K_F=(I^{\perp} :_R \langle F\rangle).
  $$

In general, we do not know which are the ideals $K_F$ that provide a minimal Gorenstein cover of a given ring. However, for a given colength, we do know a lot about the form of the ideals $K_F$ associated to a polynomial $F$ that reaches this minimum. In the following proposition, we summarize the basic results regarding ideals $K_F$ from \cite{EH18}:

\begin{proposition}
\label{KF}
Let $A=R/I$ be a local Artin algebra and $G=R/J$, with $J=\ann F$, a minimal Gorenstein cover of $A$. Then,
  \begin{enumerate}
    \item[(i)] $I^{\perp}= K_F \circ F$,
    \item[(ii)] $\gcl(A)=\ell(R/K_F)$.
  \end{enumerate}
Moreover,
$$
K_F=
\left\{
          \begin{array}{ll}
            R, &    \hbox{if \quad  } \gcl(A)=0; \\
            \m, & \hbox{if \quad } \gcl(A)=1;\\
            (L_1,\dots,L_{n-1},L_n^2), & \hbox{if \quad } \gcl(A)=2,
          \end{array}
        \right.
$$
\noindent
where $L_1,\dots,L_n$ are suitable independent linear forms in $R$.
\end{proposition}

\begin{remark}
Note that whereas in the case of colength 1 the ideal $K_F$ does not depend on the particular choice of $F$, this is no longer true for higher colengths. For colength higher that 2, things get more complicated since the $K_F$ can even have different analytic type. The simplest example is colength 3, where we have 2 possible non-isomorphic $K_F$'s: $(L_1,\dots,L_{n-1},L_n^3)$ and $(L_1,\dots,L_{n-2},L_{n-1}^2,L_{n-1}L_n,L_n^2)$. Therefore, although it is certainly true that $F\in\int_{K_F} I^\perp$, it will not be useful as a condition to check if $A$ has a certain Gorenstein colength.
\end{remark}

The dependency of the integral on $F$ can be removed by imposing only the condition $F\in\int_{\m^t} I^\perp$, for a suitable integer $t$. Later on we will see how to use this condition to find a minimal cover, but we first need to dig deeper into the structure of the integral of a module with respect to a power of the maximal ideal. The following result permits an iterative approach:

\begin{lemma}\label{str} Let $M$ be a finitely generated sub-$R$-module of $S$ and $d\geq 1$, then
$$\int_{\m}\left(\int_{\m^{d-1}}M\right)=\int_{\m^d} M.$$
\end{lemma}
\begin{proof}
Let us prove first the inclusion $\int_{\m}\left(\int_{\m^{d-1}}M\right)\subseteq \int_{\m^d} M$. Take $\Lambda\in\int_{\m}\left(\int_{\m^{d-1}}M\right)$, then $\m\circ\Lambda\subseteq \int_{\m^{d-1}}M$ and hence $\m^d\circ\Lambda=\m^{d-1}\circ\left(\m\circ\Lambda\right)\subseteq M$. Therefore, $\Lambda\in\int_{\m^d} M$.
To prove the reverse inclusion, consider $\Lambda\in\int_{\m^d} M$, that is, $\m^{d-1}\circ\left(\m\circ\Lambda\right)=\m^d\circ\Lambda\subseteq M$. In other words, $\m\circ\Lambda\subseteq\int_{\m^{d-1}}M$ and $\Lambda\in\int_{\m}\left(\int_{\m^{d-1}}M\right)$.
\end{proof}

Since $\int_{\m^t}M$ is a finitely dimensional $\res$-vector space that can be obtained by integrating $t$ times $M$ with respect to $\m$, we can also consider a basis of $\int_{\m^t}M$ which is built by extending the previous basis at each step.
\begin{definition}\label{adapted}
Let $M$ be a finitely generated sub-$R$-module of $S$.
Given an integer $t$, we denote by $h_i$ the dimension of the $\res$-vector space $\int_{\m^i}M/\int_{\m^{i-1}}M$, $i=1,\cdots,t$.
An adapted  $\res$-basis of $\int_{\m^t}M/M$ is a $\res$-basis $\overline{F}_j^i$, $i=1,\cdots, t$, $j=1,\cdots,h_i$, of
$\int_{\m^t}M/M$ such that
$F_1^i,\cdots, F_{h_i}^i\in \int_{\m^i}M$ and their cosets in $\int_{\m^i}M/\int_{\m^{i-1}}M$
form a $\res$-basis, $i=1\cdots,t$.

\noindent
Let $A=R/I$ be an Artin ring, we denote by $\mathcal{L}_{A,t}$ the $R$-module $\int_{\m^t}I^\perp/I^\perp$.
\end{definition}

The following proposition is meant to overcome the obstacle of non-uniqueness of the ideals $K_F$:

\begin{proposition}\label{F} Given a ring $A=R/I$ of Gorenstein colength $t$ and a minimal Gorenstein cover $G=R/\ann F$ of $A$,
\begin{enumerate}
\item[(i)] $F\in\int_{\m^t}I^\perp$;
\item[(ii)] for any $H\in\int_{\m^t}I^\perp$, the condition $I^\perp\subset\langle H\rangle$ does not depend on the representative of the class $\overline{H}$ in $\mathcal{L}_{A,t}$.
\end{enumerate}
In particular, any $F'\in \int_{\m^t}I^\perp$ such that $\overline{F'}=\overline{F}$ in $\mathcal{L}_{A,t}$ defines the same minimal Gorenstein cover $G=R/\ann F$.
\end{proposition}
\begin{proof}
\noindent
(i) By \cite[Proposition 3.8]{EH18}, we have $\gcl(A)=\ell(R/K_F)$, where $K_F\circ F=I^\perp$ for any polynomial $F$ that generates a minimal Gorenstein cover $G=R/\ann F$ of $A$. From the definition of integral we have $F\in\int_{K_F} I^\perp$. Since $\ell(R/K_F)=t$, then $\socdeg(R/K_F)\leq t-1$. Indeed, the extremal case corresponds to the most expanded Hilbert function $\lbrace 1,1,\dots,1\rbrace$, that is, a stretched algebra (see \cite{Sal79c},\cite{EV08}). Then $\HF_{R/K_F}(i)=0$, for any $i\geq t$, regardless of the particular form of $K_F$, and hence $\m^t\subset K_F$. Therefore,
$$F\in\int_{K_F} I^\perp\subset\int_{\m^t} I^\perp.$$
\noindent
(ii) Consider a polynomial $H\in\int_{\m^t}I^\perp$ such that $I^\perp\subset\langle H\rangle$. By \cite[Proposition 3.8]{EH18}, $K_H\circ H=I^\perp$.
Consider $H'\in\int_{\m^t}I^\perp$ such that $\overline{H}=\overline{H'}$ in $\mathcal{L}_{A,t}$, so $H=H'+G$ for some $G\in I^\perp$. We want to prove that
\begin{equation}\label{nak}
K_H\circ H'+\m\circ I^\perp=K_H\circ H+\m\circ I^\perp=I^\perp.
\end{equation}
\noindent
The second equality is direct from $K_H\circ H=I^\perp$. Let us check the first.
Take $h\circ H'+\m\circ I^\perp\in K_H\circ H'+\m\circ I^\perp$, with $h\in K_H \subset \m$,
$$h\circ H'+\m\circ I^\perp=h\circ H-h\circ G+\m\circ I^\perp=h\circ H+\m\circ I^\perp\subset  K_H\circ H+\m\circ I^\perp.$$
\noindent
The same argument holds for the reverse inclusion. Therefore, \Cref{nak} holds and we can apply Nakayama's lemma to get $K_H\circ H'=I^\perp$. Hence $I^\perp\subset\langle H'\rangle$.
In particular, $\langle H'\rangle=\langle H\rangle$. Indeed, since $H'=H-G$ and $\langle G\rangle\subset \langle I^\perp\rangle\subset \langle H\rangle$, then $H'\in \langle H\rangle + \langle G\rangle=\langle H\rangle$ and a similar argument gives $H\in\langle H'\rangle$.
\end{proof}

Observe that the proposition says that, although not all $F\in\int_{\m^t} I^\perp$ correspond to covers $G=R/\ann F$ of $A=R/I$, if $F$ is actually a cover, then any $F'\in\int_{\m^t} I^\perp$ such that
$\overline{F'}=\overline{F}\in  \mathcal{L}_{A,t}$ provides the exact same cover. That is, $\langle F'\rangle=\langle F\rangle$.

\begin{corollary}\label{cor}
Let $A=R/I$ be an Artin ring  of Gorenstein colength $t$ and let $\lbrace\overline{F}_j^i\rbrace_{1\leq i\leq t,1\leq j\leq h_i}$ be an adapted $\res$-basis of $\mathcal{L}_{A,t}$.
Given a minimal Gorenstein cover $G=R/J$ there is  a generator $F$ of $J^\perp$ such that
$F$ can be written as
$$F=a_1^1 F_1^1+\dots+a_{h_1}^1F_{h_1}^1+\dots+a_1^t F_1^t+\dots+a_{h_t}^t F_{h_t}^t\in\int_{\m^t}I^\perp,\,a_i^j\in\res.$$
\end{corollary}

\begin{proof}
In $\mathcal{L}_{A,t}$ we have $\overline{F}=\sum_{i=1}^t\sum_{j=1}^{h_j}a_j^i\overline{F_j^i}$
and hence $F=\sum_{i=1}^t\sum_{j=1}^{h_i}a_j^iF_j^i+G$ with $G\in I^\perp.$
By \Cref{F}, any representative of the class $\overline{F}$ provides the same Gorenstein cover. In particular, we can take $G=0$ and we are done.
\end{proof}

Our goal now is to compute the integrals of the inverse system with respect to powers of the maximal ideal. Rephrasing it in a more general manner: we want an effective computation of $\int_{\m^k} M$, where $M\subset S$ is a sub-$R$-module of $S$ and $k\geq 1$.

Recall that, via Macaulay's duality, we have $I^\perp=M$, where $I=\ann M$ is an ideal in $R$. Therefore, the most natural approach is to integrate $M$ in a similar way as $I$ is integrated in \Cref{EM} by Elkadi-Mourrain but removing the condition of orthogonality with respect to the generators of the ideal $I$ (\Cref{cond2} of \Cref{EM}). Without this restriction we will be allowed to go beyond the inverse system $I^\perp=M$ and up to the integral of $M$ with respect to $\m$. The proof we present is very similar to the proof of Theorem 7.36 in \cite{Mou96} but we reproduce it below for the sake of completeness and to show the use of the contraction structure.

\begin{theorem}\label{propint}
Consider a sub-$R$-module $M$ of $S$ and let $\lbrace b_1,\dots,b_s\rbrace$ be a $\res$-basis of $M$. Let $\Lambda\in S$ be a polynomial with no constant terms. Then $\Lambda\in\int_{\m} M$ if and only if
\begin{equation}\label{prop}
\Lambda=\sum_{j=1}^s\lambda_j^1\int_1 b_j\vert_{y_2=\cdots=y_n=0}+\sum_{j=1}^s\lambda_j^2\int_2 b_j\vert_{y_3=\cdots=y_n=0}+\dots+\sum_{j=1}^s\lambda_j^n\int_n b_j,\quad\lambda_j^k\in\res,
\end{equation}
such that
\begin{equation}\label{cond1}
\sum_{j=1}^s\lambda_j^k (x_l\circ b_j)-\sum_{j=1}^s\lambda_j^l(x_k\circ b_j)=0, 1\leq k<l\leq n.
\end{equation}
\end{theorem}

\begin{proof}
Consider a polynomial $\Lambda$ in $\int_\m M$ with no constant term. Observe that we have a unique decomposition $\Lambda=\sum_{l=1}^n\Lambda_l$ such that $\Lambda_l$ is a polynomial in $\res[y_l,\dots,y_n]\backslash\res[y_{l+1},\dots,y_n]$. By definition, $x_1\circ\Lambda_1=x_1\circ\Lambda$ is in $M$, hence $x_1\circ\Lambda_1=\sum_{j=1}^s\lambda_j^1b_j$ for some unique scalars $\lambda_j^1$ in $\res$. Note that each $\Lambda_l$ is a multiple of $y_l$. By \Cref{r1},

$$\Lambda_1=\int_1 x_1\circ\Lambda_1=\sum_{j=1}^s\lambda_j^1\int_1b_j.$$

Again, $x_2\circ\Lambda=x_2\circ\Lambda_1+x_2\circ\Lambda_2$ is in $M$, hence there exist unique scalars $\lambda_j^2$ in $\res$ such that $x_2\circ\Lambda=\sum_{j=1}^s\lambda_j^2b_j$. It can be checked that $\int_2x_2\circ\Lambda_1=\Lambda_1-\Lambda_1\mid_{y_2=0}$. Then
$$\Lambda_2=\int_2 x_2\circ\Lambda_2=\int_2x_2\circ\Lambda-\int_2x_2\circ\Lambda_1=\sum_{j=1}^{s}\lambda_j^2\int_2b_j-\left(\Lambda_1-\Lambda_1\vert_{y_2=0}\right).$$

%\noindent
% and $$\Lambda_1+\Lambda_2=\sum_{j=1}^s\lambda_j^1\int_1b_j\vert_{y_2=0%}+\sum_{j=1}^s\lambda_j^2\int_2b_j.$$

Similarly, for any $1\leq l\leq n$, we can obtain
\begin{equation}\label{lambda1}
\Lambda_l=\sum_{j=1}^s\lambda_j^l\int_l b_j-\left(\sigma_{l-1}-\sigma_{l-1}\mid_{y_l=0}\right),
\end{equation}
\noindent
where
\begin{equation}\label{sigma1}
\sigma_k=\sum_{l=1}^k\Lambda_l=\sum_{j=1}^s\lambda_j^1\int_1b_j\mid_{y_2=\dots=y_k=0}+\sum_{j=1}^s\lambda_j^2\int_2b_j\mid_{y_3=\dots=y_k=0}+\dots+\sum_{j=1}^s\lambda_j^k\int_kb_j,
\end{equation}
\noindent
for any $1\leq k\leq n$ and $\sigma_0=0$.

Since $\Lambda=\sigma_n$, we get (\ref{prop}). We want to prove now that (\ref{cond1}) holds. Since $\Lambda_l\in\res[y_l,\dots,y_n]$, then $x_k\circ\Lambda_l=0$ for $1\leq k<l\leq n$. Hence contracting (\ref{lambda1}) first by $x_k$ and then by $x_l$ we get
\begin{equation}\label{sigma3}
\sum_{j=1}^s\lambda_j^l(x_k\circ b_j)=x_l\circ\left(x_k\circ\sigma_{l-1}\right).
\end{equation}
\noindent
On one hand, for $k<l$, $x_k\circ\sigma_{l-1}=x_k\circ(\sum_{i=1}^k\Lambda_i)=x_k\circ\sigma_k$. On the other hand, when contracting (\ref{sigma1}) by $x_k$, the first $k-1$ terms vanish:
$$x_k\circ\sigma_k=\sum_{j=1}^s\lambda_j^k\left(x_k\circ\int_k b_j\right)=\sum_{j=1}^s\lambda_j^k b_j.$$
\noindent
Therefore, we can rewrite (\ref{sigma3}) as $\sum_{j=1}^s\lambda_j^l(x_k\circ b_j)=\sum_{j=1}^s\lambda_j^k(x_l\circ b_j),$ hence (\ref{cond1}) is satisfied.

Conversely, we want to know if every element of the form (\ref{prop}) satisfying (\ref{cond1}) is in $\int_{\m}M$. It is enough to prove that $x_k\circ\Lambda\in M$ for any $1\leq k\leq n$. Let us then contract (\ref{prop}) by $x_k$ for any $1\leq k\leq n$:
$$x_k\circ\Lambda=\sum_{j=1}^s\lambda_j^kb_j\mid_{y_{k+1}=\dots=y_n=0}+\sum_{j=1}^s\lambda_j^{k+1}\int_{k+1} x_k\circ b_j\mid_{y_{k+2}=\dots=y_n=0}+\dots+\sum_{j=1}^s\lambda_j^n\int_n x_k\circ b_j.$$
\noindent
The $l$-primitive of (\ref{cond1}), for any $k<l\leq n$, gives
$$\sum_{j=1}^s\lambda_j^k \int_l x_l\circ b_j=\sum_{j=1}^s\lambda_j^l\int_l x_k\circ b_j.$$
Hence
$$x_k\circ\Lambda=\sum_{j=1}^s\lambda_j^k\left(b_j\mid_{y_{k+1}=\dots=y_n=0}+\int_{k+1} x_{k+1}\circ b_j\mid_{y_{k+2}=\dots=y_n=0}+\dots+\int_n x_n\circ b_j\right).$$
\noindent
It can be proved that the expression in the parenthesis is exactly $b_j$ for any $1\leq j\leq n$, hence $x_k\circ\Lambda=\sum_{j=1}^s\lambda_j^kb_j$ and we are done.
\end{proof}

\noindent
From the previous theorem and \Cref{str} the next corollary follows directly.

\begin{corollary}\label{thmint}
Consider a sub-$R$-module $M$ of $S$ and $d\geq 1$. Let $\lbrace b_1,\dots,b_{t_{d-1}}\rbrace$ be a $\res$-basis of $\int_{\m^{d-1}} M$ and let $\Lambda$ be a polynomial with no constant terms. Then $\Lambda\in\int_{\m^d} M$ if and only if it is of the form
\begin{equation}\label{thm2}
\Lambda=\sum_{j=1}^{t_{d-1}}\lambda_j^1\int_1 b_j\vert_{y_2=\cdots=y_n=0}+\sum_{j=1}^{t_{d-1}}\lambda_j^2\int_2 b_j\vert_{y_3=\cdots=y_n=0}+\dots+\sum_{j=1}^{t_{d-1}}\lambda_j^n\int_n b_j,\quad\lambda_j^k\in\res,
\end{equation}
such that
\begin{equation}\label{cond}
\sum_{j=1}^{t_{d-1}}\lambda_j^k (x_l\circ b_j)-\sum_{j=1}^{t_{d-1}}\lambda_j^l(x_k\circ b_j)=0,\quad 1\leq k<l\leq n.
\end{equation}
\end{corollary}

\begin{remark} Note that, using the notations of \Cref{EM}, it can be proved that
$$\mathcal{D}_d=I^\perp\cap\int_{\m}\mathcal{D}_{d-1},$$

\noindent
for any $1<d\leq s$. Indeed, \Cref{EM} says that any element $\Lambda\in\mathcal{D}_d$ is of the form of \Cref{thm2}, and because of \Cref{thmint}, we know that it satisfies \Cref{cond}. Hence, by \Cref{propint}, $\Lambda\in\int_{\m}\mathcal{D}_{d-1}$. Since $\Lambda\in\mathcal{D}_d=I^\perp\cap S_{\leq d}$, then $\Lambda\in I^\perp\cap\int_{\m}\mathcal{D}_{d-1}$.
Conversely, any element $\Lambda$ in $\left(\int_{\m}\mathcal{D}_{d-1}\right)\cap I^\perp$ satisfies, in particular, $\m\circ\Lambda\subseteq \mathcal{D}_{d-1}=I^\perp\cap S_{\leq d-1}$. Therefore $\deg\left(\m\circ\Lambda\right)\leq d-1$ and hence $\deg\Lambda\leq d$. Since $\Lambda\in I^\perp$, then $\Lambda\in I^\perp\cap S_{\leq d}=\mathcal{D}_d$.
\end{remark}

\bigskip
We end this section by considering the low Gorenstein colength cases.

\subsection{Teter rings}

Let us remind that Teter rings are those $A=R/I$ such that $A\cong G/\soc(G)$ for some Artin Gorenstein ring $G$. In \cite{ES17}, the authors prove that $\gcl(A)=1$ whenever $\embd(A)\geq 2$. They are a special case to deal with because the $K_F$ associated to any generator $F\in S$ of a minimal cover is always the maximal ideal. We provide some additional criteria to characterize such rings:

\begin{proposition}\label{propTeter} Let $A=R/I$ be a non-Gorenstein local Artin ring of socle degree $s\geq 1$ and let $\lbrace\overline{F}_j\rbrace_{1\leq j\leq h}$ be an adapted $\res$-basis of $\mathcal{L}_{A,1}$. Then $\gcl(A)=1$ if and only if there exist a polynomial $F=\sum_{j=1}^h a_jF_j\in\int_{\m}I^\perp$, $a_j\in\res$, such that $\dim_\res(\m\circ F)=\dim_\res I^\perp$.
\end{proposition}
\begin{proof} The first implication is straightforward from \Cref{cor} and Teter rings characterization in \cite{ES17}.
Reciprocally, if $F\in\int_{\m}I^\perp$, then $\m\circ F\subset I^\perp$ by definition, and from the equality of dimensions, it follows that $\m\circ F=I^\perp$. Therefore, $0<\gcl(A)\leq\ell(R/\m)=1$ and we are done.
\end{proof}

\begin{example}\label{Ex3} Recall \Cref{Ex1} with $I^\perp=\langle y_1y_2,y_3^3\rangle$ and $\int_\m I^\perp=\langle y_1^2,y_1y_2,y_1y_3,y_2^2,y_2y_3,y_3^4\rangle$. Then $\overline{y_1^2},\overline{y_1y_3},\overline{y_2^2},\overline{y_2y_3},\overline{y_3^4}$ is a $\res$-basis of $\mathcal{L}_{A,1}$. As a consequence of \Cref{propTeter}, $A$ is Teter if and only if there exists a polynomial
$$F=a_1y_1^2+a_2y_1y_3+a_3y_2^2+a_4y_2y_3+a_5y_3^4$$
such that $\m\circ F=I^\perp$. But $\m\circ F=\langle a_1y_1+a_2y_3,a_3y_2+a_4y_3,a_2y_1+a_4y_2+a_5y_3^3\rangle$ and clearly $y_1y_2$ does not belong here. Therefore, $\gcl(A)>1$.
\end{example}

\subsection{Gorenstein colength 2}

By \cite{EH18}, we know that an Artin ring $A$ of socle degree $s$ is of Gorenstein colength 2 if and only if there exists a polynomial $F$ of degree $s+1$ or $s+2$ such that $K_F\circ F=I^\perp$, where $K_F=(L_1,\dots,L_{n-1},L_n^2)$ and $L_1,\dots,L_n$ are suitable independent linear forms.

Observe that a completely analogous characterization to the one we did for Teter rings is not possible. If $A=R/I$ has Gorenstein colength 2, by \Cref{cor}, there exists $F=\sum_{i=1}^2\sum_{j=1}^{h_i}a_j^iF_j^i\in\int_{\m^2}I^\perp$, where $\lbrace\overline{F^i_j}\rbrace_{1\leq i\leq 2,1\leq j\leq h_i}$ is a $\res$-basis of $\mathcal{L}_{A,2}$, that generates a minimal Gorenstein cover of $A$ and then trivially $I^\perp\subset\langle F\rangle$. However, the reverse implication is not true.

\begin{example} Consider $A=R/\m^3$, where $R$ is the ring of power series in 2 variables, and consider $F=y_1^2y_2^2$. It is easy to see that $F\in\int_{\m^2}I^\perp=S_{\leq 4}$ and $I^\perp\subset\langle F\rangle$. However, it can be proved that $\gcl(A)=3$ using \cite[Corollary 3.3]{Ana09}. Note that $K_F=\m^2$ and hence $\ell(R/K_F)=3$.
\end{example}

Therefore, given $F\in\int_{\m^2}I^\perp$, the condition $I\subset\langle F\rangle$ is not sufficient to ensure that $\gcl(A)=2$. We must require that $\ell(R/K_F)=2$ as well.

\begin{proposition}\label{gcl2} Given a non-Gorenstein non-Teter local Artin ring $A=R/I$, $\gcl(A)=2$ if and only if there exist a polynomial $F=\sum_{i=1}^2\sum_{j=1}^{h_i} a_j^iF_j^i\in\int_{\m^2}I^\perp$ such that $\lbrace\overline{F_j^i}\rbrace_{1\leq i\leq 2,1\leq j\leq h_i}$ is an adapted $\res$-basis of $\mathcal{L}_{A,2}$ and $(L_1,\dots,L_{n-1},L_n^2)\circ F=I^\perp$ for suitable independent linear forms $L_1,\dots,L_n$.
\end{proposition}

\begin{proof} We will only prove that if $F$ satisfies the required conditions, then $\gcl(A)=2$. By definition of $K_F$, if $(L_1,\dots,L_{n-1},L_n^2)\circ F=I^\perp$, then $(L_1,\dots,L_{n-1},L_n^2)\subseteq K_F$. Again by \cite{EH18}, $\gcl(A)\leq \ell(R/K_F)$ and hence $\gcl(A)\leq \ell\left(R/(L_1,\dots,L_{n-1},L_n^2)\right)=2$. Since $\gcl(A)\geq 2$ by hypothesis, then $\gcl(A)=2$. The converse implication follows from \Cref{KF}.
\end{proof}

\begin{example} Recall the ring $A=R/I$ in \Cref{Ex3}. Since
$$\int_{\m^2}I^\perp=\langle y_1^3,y_1^2y_2,y_1y_2^2,y_2^3,y_1^2y_3,y_1y_2y_3,y_2^2y_3,y_1y_3^2,y_2y_3^3,y_3^5\rangle$$
and $\gcl(A)>1$, its Gorenstein colength is 2 if and only if there exist some
$$F\in\langle  y_1^2,y_1y_2,y_1y_3,y_2^2,y_2y_3,y_3^4,y_1^3,y_1^2y_2,y_1y_2^2,y_2^3,y_1^2y_3,y_1y_2y_3,y_2^2y_3,y_1y_3^2,y_2y_3^3,y_3^5\rangle_\res$$
such that $(L_1,\dots,L_{n-1},L_n^2)\circ F=I^\perp$. Consider $F=y_3^4+y_1^2y_2$, then
$$(x_1,x_2^2,x_3)\circ F=\langle x_1\circ F,x_2^2\circ F,x_3\circ F\rangle=\langle y_1y_2,y_3^3\rangle$$
and hence $\gcl(A)=2$.
\end{example}

%%%%%%%%%%%%%%%%%%%%%%%%%%%%%%%%%%%%%%%%%%%%%%%%%%%%%%%%%%%%%%%%%%%%%%%%%%%%%%%%%%%%%%%%%
%%%%%%%%%%%%%%%%%%%%%%%%%%%%%%%%%%%%%%%%%%%%%%%%%%%%%%%%%%%%%%%%%%%%%%%%%%%%%%%%%%%%%%%%%
%%%%%%%%%%%%%%%%%%%%%%%%%%%%%%%%%%%%%%%%%%%%%%%%%%%%%%%%%%%%%%%%%%%%%%%%%%%%%%%%%%%%%%%%%
\section{Minimal Gorenstein covers varieties}

We are now interested in providing a geometric interpretation of the set of all minimal Gorenstein covers $G=R/J$ of a given local Artin $\res$-algebra $A=R/I$. From now on, we will assume that $\res$ is an algebraically closed field. The following result is well known and it is an easy linear algebra exercise.

\begin{lemma}
\label{semic}
Let $\varphi_i:\res^a 	\longrightarrow \res^b$, $i=1\cdots,r$, be a family of Zariski continuous maps.
 Then the function $\varphi^*:\res^a\longrightarrow \mathbb N$ defined by
$\varphi^*(z)=\dim_{\res} \langle \varphi_1(z),\cdots, \varphi_r(z)\rangle_{\res}$ is lower semicontinous, i.e. for all $z_0 \in \res^a$ there is a Zariski open set
$z_0\in U \subset \res^a$ such that for all $z\in U$ it holds
$\varphi^*(z)\geq \varphi^*(z_0)$.
\end{lemma}

\begin{theorem}\label{ThMGC}
Let $A=R/I$ be an Artin ring of Gorenstein colength $t$.
There exists a quasi-projective sub-variety $MGC^n(A)$, $n=\dim(R)$, of
$\mathbb P_{\res}\left(\mathcal{L}_{A,t}\right)$
whose set of closed points are the points $[\overline{F}]$, $\overline{F}\in \mathcal{L}_{A,t}$,
such that $G=R/\ann F$ is a minimal Gorenstein cover of $A$.
\end{theorem}
\begin{proof}
Let $E$ be a sub-$\res$-vector space of $\int_{\m^t}I^\perp$ such that
$$
\int_{\m^t}I^\perp \cong E\oplus I^{\perp},
$$
we identify $\mathcal{L}_{A,t}$ with $E$. From \Cref{F}, for all minimal Gorenstein cover $G=R/\ann F$ we may assume that $F\in E$. Given $F\in E$, the quotient $G=R/\ann F$ is a minimal cover of $A$ if and only if the following two numerical conditions hold:
\begin{enumerate}
\item $\dim_{\res}(\langle F\rangle)= \dim_{\res}A+t$, and
\item $\dim_{\res}(I^{\perp}+ \langle F \rangle) =\dim_{\res}\langle F \rangle$.
\end{enumerate}

\noindent
Define the family of Zariski continuous maps $\lbrace\varphi_{\alpha}\rbrace_{\vert\alpha\vert\leq\deg F}$, $\alpha\in\mathbb{N}^n$, where
$$\begin{array}{rrcl}
\varphi_{\alpha}: & E & \longrightarrow & E\\
& F & \longmapsto & x^\alpha\circ F\\
\end{array}$$
\noindent
In particular, $\varphi_0=Id_R$.
We write
$$\begin{array}{rrcl}
\varphi^\ast: & E & \longrightarrow & \mathbb{N}\\
& F & \longmapsto & \dim_\res\langle x^\alpha\circ F,\vert\alpha\vert\leq\deg F\rangle_\res
\end{array}$$

\noindent
Note that $\varphi^\ast(F)=\dim_\res \langle F\rangle$ and, by \Cref{semic}, $\varphi^\ast$ is a lower semicontinuous map. Hence $U_1=\lbrace F\in E\mid \dim_\res\langle F\rangle\geq\dim_\res A+t\rbrace$ is an open Zariski set in $E$. Using the same argument,
$U_2=\lbrace F\in E\mid \dim_\res\langle F\rangle\geq\dim_\res A+t+1\rbrace$
is also an open Zariski set in $E$ and hence $Z_1=E\backslash U_2$ is  a Zariski closed set such that $\dim_\res\langle F\rangle\leq\dim_\res A+t$ for any $F\in Z_1$.
Then $Z_1\cap U_1=\lbrace F\in E\mid \dim_\res\langle F\rangle=\dim_\res A+t\rbrace$ is a locally closed set.

Let $G_1,\cdots,G_e$ be a $\res$-basis of $I^{\perp}$ and consider the constant map
$$\begin{array}{rrcl}
\psi_{i}: & E & \longrightarrow & E\\
& F & \longmapsto & G_i\\
\end{array}$$
for any $i=1,\cdots,e$.
By \Cref{semic},

$$\begin{array}{rrcl}
\psi^\ast: & E & \longrightarrow & \mathbb{N}\\
& F & \longmapsto & \dim_\res \langle\lbrace x^{\alpha}\circ F\rbrace_{\vert\alpha\vert\leq\deg F}, G_1,\dots, G_e\rangle_\res=\dim_\res\left(\langle F\rangle+I^\perp\right)
\end{array}$$

\noindent
is a lower semicontinuous map. Using an analogous argument, we can prove that $T=\lbrace F\in E\mid \dim_\res(I^\perp+\langle F\rangle)=\dim_\res A+t\rbrace$ is a locally closed set. Therefore,
$$W=(Z_1\cap U_1)\cap T=\lbrace F\in E\mid \dim_\res A+t=\dim_\res(I^\perp+\langle F\rangle)=\dim_\res\langle F\rangle\rbrace$$
\noindent
is a locally closed subset of $E$ whose set of closed points are all the $F$ in $E$ satisfying $(1)$ and $(2)$, i.e. defining a minimal Gorenstein cover $G=R/\ann F$ of $A$.

Moreover, since $\langle F\rangle=\langle \lambda F\rangle$ for any $\lambda\in\res^\ast$, conditions $(1)$ and $(2)$ are invariant under the multiplicative action of $\res^*$ on $F$ and hence
$MGC^n(A)=\mathbb P_{\res}(W)\subset \mathbb P_{\res}(E)=\mathbb P_{\res}\left(\mathcal{L}_{A,t}\right)$.
\end{proof}

Recall that the embedding dimension of $A$ is $\embd(A)=\dim_\res\m/(\m^2+I)$.

\begin{proposition}
Let $G$ be a minimal Gorenstein cover of $A$.
Then
$$
\embd(G)\le \tau(A)+\gcl(A)-1.
$$
\end{proposition}
\begin{proof}
Set $A=R/I$ such that $\embd(A)=\dim R=n$. Consider the power series ring $R'$ of dimension $n+t$ over $\res$ for some $t\geq 0$ such that $G=R'/J'$ with $\embd(G)=\dim R'$. See \cite{EH18} for more details on this construction. We denote by $\m$ and $\m'$ the maximal ideals of $R$ and $R'$, respectively, and consider $K_{F'}=(I^\perp:_{R'} F')$. 
From \Cref{KF}$.(i)$, it is easy to deduce that $K_{F'}/(\m K_{F'}+J')\simeq I^\perp/(\m\circ I^\perp)$. Hence $\tau(A)=\dim_\res K_{F'}/(\m K_{F'}+J')$ by \cite[Proposition 2.6]{ES17}. Then
$$\embd(G)+1=\dim_\res R'/(\m')^2\leq \dim_\res R'/(\m K_{F'}+J')=\gcl(A)+\tau(A),$$
where the last equality follows from \Cref{KF}$.(ii)$.
\end{proof}

\begin{definition}\label{DefMGC}
Given an Artin ring $A=R/I$, the variety $MGC(A)=MGC^n(A)$, with $n=\tau(A)+\gcl(A)-1$, is called the minimal Gorenstein cover variety associated to $A$.
\end{definition}

\begin{remark}\label{RemMGC}
Let us recall that in \cite{EH18} we proved that for low Gorenstein colength of $A$, i.e. $\gcl(A)\le 2$, $\embd(G)=\embd(A)$ for any minimal Gorenstein cover $G$ of $A$. In this situation we can consider $MGC(A)$ as the variety $MGC^n(A)$ with $n=\embd(A)$.
\end{remark}

Observe that this notion of minimal Gorenstein cover variety generalizes the definition of Teter variety introduced in \cite{ES17}, which applies only to rings of Gorenstein colength 1, to any arbitrary colength.

%%%%%%%%%%%%%%%%%%%%%%%%%%%%%%%%%%%%%%%%%%%%%%%%%%%%%%%%%%%%%%%%%%%%%%%%%%%%%%%%%%%%%%%%
%%%%%%%%%%%%%%%%%%%%%%%%%%%%%%%%%%%%%%%%%%%%%%%%%%%%%%%%%%%%%%%%%%%%%%%%%%%%%%%%%%%%%%%%
%%%%%%%%%%%%%%%%%%%%%%%%%%%%%%%%%%%%%%%%%%%%%%%%%%%%%%%%%%%%%%%%%%%%%%%%%%%%%%%%%%%%%%%%
%%%%%%%%%%%%%%%%%%%%%%%%%%%%%%%%%%%%%%%%%%%%%%%%%%%%%%%%%%%%%%%%%%%%%%%%%%%%%%%%%%%%%%%
\section{Computing $MGC(A)$ for low Gorenstein colength}\label{s5}

In this section we provide algorithms and examples to compute the variety of minimal Gorenstein covers of a given ring $A$ whenever its Gorenstein colength is 1 or 2. These algorithms can also be used to decide whether a ring has colength greater than 2, since it will correspond to empty varieties.

To start with, we provide the auxiliar algorithm to compute the integral of $I^\perp$ with respect to the $t$-th power of the maximal ideal of $R$. If there exist polynomials defining minimal Gorenstein covers of colength $t$, they must belong to this integral.

\subsection{Computing integrals of modules}

Consider a $\res$-basis $\mathbf{b}=(b_1,\dots,b_t)$ of a finitely generated sub-$R$-module $M$ of $S$ and consider $x_k\circ b_i=\sum_{j=1}^t a_j^i b_j$, for any $1\leq i\leq t$ and $1\leq k\leq n$. Let us define matrices $U_k=(a_j^i)_{1\leq j,i\leq t}$ for any $1\leq k\leq n$. Note that

$$\left(x_k\circ b_1 \cdots x_k\circ b_t\right)=
\left(b_1 \cdots b_t\right)
\left(\begin{array}{ccc}
a_1^1 & \dots & a_1^t\\
\vdots & & \vdots\\
a_t^1 & \dots & a_t^t
\end{array}\right)
.$$

\noindent
Now consider any element $h\in M$. Then
$$x_k\circ h=x_k\circ\sum_{i=1}^t h_ib_i=\sum_{i=1}^t (x_k\circ h_ib_i)=\sum_{i=1}^t (x_k\circ b_i)h_i=$$
$$=\left(x_k\circ b_1 \cdots x_k\circ b_t\right)
\left(\begin{array}{c}
h_1\\
\vdots\\
h_t
\end{array}\right)=\left(b_1 \cdots b_t\right)U_k
\left(\begin{array}{c}
h_1\\
\vdots\\
h_t
\end{array}\right),$$
\noindent
where $h_1,\dots,h_t\in\res$.

\begin{definition} Let $U_k$, $1\leq k\leq n$, be the square matrix of order $t$ such that
$$x_k\circ h=\mathbf{b}\,U_k\,\mathbf{h}^t,$$
where $\mathbf{h}=(h_1,\dots,h_t)$ for any $h\in M$, with $h=\sum_{i=1}^t h_ib_i$. We call $U_k$ the contraction matrix of $M$ with respect to $x_k$ associated to a $\res$-basis $\mathbf{b}$ of $M$.
\end{definition}

\begin{remark} Since $x_kx_l\circ h=x_lx_k\circ h$ for any $h\in M$, we have $U_kU_l=U_lU_k$, with $1\leq k<l\leq n$.
\end{remark}

In \cite{Mou96}, Mourrain provides an effective algorithm based on \Cref{EM} that computes, along with a $\res$-basis of the inverse system $I^\perp$ of an ideal $I$ of $R$, the contraction matrices $U_1,\dots,U_n$ of $I^\perp$ associated to that basis.

\begin{example} Consider $A=R/I$, with $R=\res[\![x_1,x_2]\!]$ and $I=\m^2$. Then $\lbrace 1,y_1,y_2\rbrace$ is a $\res$-basis of $I^\perp$ and $U_1,U_2$ are its contraction matrices with respect to $x_1,x_2$, respectively:
$$U_1=\left(\begin{array}{ccc}
0 & 1& 0\\
0 & 0& 0\\
0 & 0& 0
\end{array}\right),\quad
U_2=\left(\begin{array}{ccc}
0 & 0& 1\\
0 & 0& 0\\
0 & 0& 0
\end{array}\right).$$
\end{example}

Now we provide a modified algorithm based on \Cref{propint} that computes the integral of a finitely generated sub-$R$-module $M$ with respect to the maximal ideal. The algorithm can use the output of Mourrain's integration method as initial data: a $\res$-basis of $I^\perp$ and the contraction matrices associated to this basis.

\begin{algorithm}[H]
\caption[AlINT]{Compute a $\res$-basis of $\int_{\m} M$ and its contraction matrices}
\label{AlINT}
\begin{algorithmic}
\REQUIRE $D=b_1,\dots,b_t$ $\res$-basis of $M$;\\
$U_1,\dots,U_n$ contraction matrices of $M$ associated to the $\res$-basis $D$.
\ENSURE $D=b_1,\dots,b_t,b_{t+1},\dots,b_{t+h}$ $\res$-basis of $\int_{\m} M$;\\
$U'_1,\dots,U'_n$ contraction matrices of $\int_{\m} M$ associated to the $\res$-basis $D$.
\RETURN
\begin{enumerate}
\item Set $\lambda_i=(\lambda^i_1\,\cdots\,\lambda^i_t)^t$, for any $1\leq i\leq n$. Solve the system of equations
\begin{equation}\label{system}
U_k\,\lambda_l-U_l\,\lambda_k= 0 \mbox{ for any } 1\leq k<l\leq n.
\end{equation}
\item Consider a system of generators $\mathbf{H}_1,\dots,\mathbf{H}_m$ of the solutions of \Cref{system}.
\item For any $\mathbf{H}_i=\left[\lambda_1,\dots,\lambda_n\right]$, $1\leq i\leq m$, define the associated polynomial
$$\Lambda_{\mathbf{H}_i}=\displaystyle\sum_{k=1}^n\left(\sum_{j=1}^t\lambda_j^k\int_k b_j\vert_{y_{k+1}=\cdots=y_n=0}\right).$$
\item If $\Lambda_{\mathbf{H}_1}\notin\langle D\rangle_\res$, then $b_{t+1}:=\Lambda_{\mathbf{H}_1}$ and $D=D,b_{t+1}$. Repeat the procedure for $\Lambda_{\mathbf{H}_2},\dots,\Lambda_{\mathbf{H}_m}$.
\item Set $h$ as the number of new elements in $D$.
\item Define square matrices $U'_k$ of order $t+h$ and set $U'_k[i]=U_k[i]$ for $1\leq i\leq t$.
\item Compute $x_k \circ b_i=\sum_{j=1}^t\mu_j^ib_j$ for $t+1\leq i\leq t+h$ and set
$$U'_k[i]=\left(\begin{array}{cccccc}
\mu_1^i & \cdots & \mu_t^i & 0 & \cdots & 0
\end{array}\right)^t.$$
\end{enumerate}
\end{algorithmic}
\end{algorithm}

\begin{remark} Observe that the classes in $\int_\m M/M$ of the output $b_{t+1},\dots,b_{t+h}$ of \Cref{AlINT} form a $\res$-basis of $\int_\m M/M$. Moreover, since the algorithm returns the contraction matrices of $\int_\m M$, we can iterate the procedure in order to obtain a $\res$-basis of $\int_{\m^k} M$ for any $k\geq 1$. By construction, the elements of this $\res$-basis that do not belong to $M$ form an adapted $\res$-basis of $\int_{\m^k}M/M$.
\end{remark}

\begin{example}\label{ExAlINT} Consider $A=R/I$, with $R=\res[\![x_1,x_2]\!]$ and $I=\m^2$. Then $\lbrace 1,y_1,y_2,y_2^2,y_1y_2,y_1^2\rbrace$ is a $\res$-basis of $\int_{\m}I^\perp=S_{\leq 2}$ with the following contraction matrices:
$$U'_1=\left(\begin{array}{cccccc}
0 & 1& 0 & 0 & 0 & 0\\
0 & 0& 0 & 0 & 0 & 1\\
0 & 0& 0 & 0 & 1 & 0\\
0 & 0& 0 & 0 & 0 & 0\\
0 & 0& 0 & 0 & 0 & 0\\
0 & 0& 0 & 0 & 0 & 0\\
\end{array}\right),\quad
U'_2=\left(\begin{array}{cccccc}
0 & 0& 1 & 0 & 0 & 0\\
0 & 0& 0 & 0 & 1 & 0\\
0 & 0& 0 & 1 & 0 & 0\\
0 & 0& 0 & 0 & 0 & 0\\
0 & 0& 0 & 0 & 0 & 0\\
0 & 0& 0 & 0 & 0 & 0\\
\end{array}\right).$$
\end{example}

\bigskip

\subsection{Computing $MGC(A)$ for Teter rings}

The following algorithm provides a method to decide whether a non-Gorenstein ring $A=R/I$ has colength 1 and, if this is the case, it explicitly computes its $MGC(A)$.

Let us consider a non-Gorenstein local Artin ring $A=R/I$ of socle degree $s$. Fix a $\res$-basis $b_1,\dots,b_t$ of $I^\perp$ and consider a polynomial $F=\sum_{j=1}^h a_jF_j\in \int_{\m}I^\perp$, where $\overline{F}_1,\dots,\overline{F}_h$ is an adapted $\res$-basis of $\mathcal{L}_{A,1}$. According to \Cref{propTeter}, $F$ corresponds to a minimal Gorenstein cover if and only if $\dim_\res(\m\circ F)=t$. Therefore, we want to know for which values of $a_1,\dots,a_h$ this equality holds.

Note that $\deg F\leq s+1$ and $x_kx_l \circ F=x_lx_k\circ F$. Then $\m\circ F=\langle x^{\alpha}\circ F: 1\leq\vert \alpha\vert\leq s+1\rangle_\res$. Moreover, by definition of $F$, each $x^{\alpha}\circ F\in I^\perp$, hence $x^{\alpha}\circ F=\sum_{j=1}^t\mu_{\alpha}^jb_j$ for some $\mu_{\alpha}^j\in\res$.

Consider the matrix $A=(\mu_{\alpha}^j)_{1\leq\vert \alpha\vert\leq s+1,\,1\leq j\leq t}$, whose rows are the contractions $x^{\alpha}\circ F$ expressed in terms of the $\res$-basis $b_1,\dots,b_t$ of $I^\perp$. The rows of $A$ are a system of generators of $\m\circ F$ as $\res$-vector space, hence $\dim_\res(\m\circ F)<t$ if and only if all order $t$ minors of $A$ vanish. Let $\mathfrak{a}$ be the ideal generated by all order $t$ minors $p_1,\dots,p_r$ of $A$. Note that the entries of matrix $A$ are homogeneous polynomials of degree 1 in $\res[a_1,\dots,a_h]$. Hence $\mathfrak{a}$ is generated by homogeneous polynomials of degree $t$ in $\res[a_1,\dots,a_h]$. Therefore, we can view the projective algebraic set
$$\mathbb{V}_+(\mathfrak{a})=\lbrace [a_1:\dots:a_h]\in\mathbb{P}_\res^{h-1}\mid p_i(a_1,\dots,a_h)=0,\, 1\leq i\leq r\rbrace,$$
\noindent
as the set of all points that do not correspond to Teter covers. We just proved the following result:

\begin{theorem}
Let $A=R/I$ be an Artin ring with $\gcl(A)=1$, $h=\dim_\res\mathcal{L}_{A,1}$ and $\mathfrak{a}$ be the ideal of minors previously defined. Then
$$MGC(A)=\mathbb{P}_\res^{h-1}\backslash\mathbb{V}_+(\mathfrak{a}).$$
Moreover, for any non-Gorenstein Artin ring $A$, $\gcl(A)=1$ if and only if $\mathfrak{a}\neq 0$.
\end{theorem}

\begin{proof}
The first part is already proved. On the other hand, if $\mathfrak{a}=0$, then $\mathbb{V}_+(\mathfrak{a})=\mathbb{P}_\res^{h-1}$ and $MGC(A)=\emptyset$. In other words, there exist no Teter covers, hence $\gcl(A)>1$.
\end{proof}

\begin{algorithm}[H]
\caption{Compute the Teter variety of $A=R/I$ with $n\geq 2$}
\label{AlMGC1}
\begin{algorithmic}
\REQUIRE
$s$ socle degree of $A=R/I$;\\
$b_1,\dots,b_t$ $\res$-basis of the inverse system $I^\perp$;\\
$F_1,\dots,F_h$ adapted $\res$-basis of $\mathcal{L}_{A,1}$;\\
$U_1,\dots,U_n$ contraction matrices of $\int_\m I^\perp$.
\ENSURE
%$F=a_1F_1+\dots+a_hF_h$ polynomial defining a minimal Gorenstein cover;\\
ideal $\mathfrak{a}$ such that $MGC(A)=\mathbb{P}_\res^{h-1}\backslash\mathbb{V}_+(\mathfrak{a})$.
\RETURN
\begin{enumerate}
\item Set $F=a_1F_1+\dots+a_hF_h$ and $\mathbf{F}=(a_1,\dots,a_h)^t$, where $a_1,\dots,a_h$ are variables in $\res$.
\item Build matrix $A=\left(\mu^\alpha_j\right)_{1\leq\vert\alpha\vert\leq s+1,1\leq j\leq t}$, where $$U^\alpha \textbf{F}=\sum_{j=1}^t\mu^\alpha_jb_j,\quad U^\alpha=U_1^{\alpha_1}\cdots U_n^{\alpha_n}.$$
\item Compute the ideal $\mathfrak{a}$ generated by all minors of order $t$ of the matrix $A$.
\end{enumerate}
\end{algorithmic}
\end{algorithm}

With the following example we show how to apply and interpret the output of the algorithm:
\begin{example}
Consider $A=R/I$, with $R=\res[\![x_1,x_2]\!]$ and $I=\m^2$ \cite[Example 4.3]{ES17}. From \Cref{ExAlINT} we gather all the information we need for the input of \Cref{AlMGC1}:
{\sc Input}:
$b_1=1,b_2=y_1,b_3=y_2$ $\res$-basis of $I^\perp$; $F_1=y^2,F_2=y_1y_2,F_3=y_1^2$ adapted $\res$-basis of $\mathcal{L}_{A,1}$; $U_1'$,$U_2'$ contraction matrices of $\int_\m I^\perp$.\\
{\sc Output}: $\rad(\mathfrak{a})=a_2^2-a_1a_3$.\\
Then $MGC(A)=\mathbb{P}^2\backslash\lbrace a_2^2-a_1a_3=0\rbrace$ and any minimal Gorenstein cover $G=R/\ann F$ of $A$ is given by a polynomial $F=a_1y^2_2+a_2y_1y_2+a_3y_1^2$ such that $a_2^2-a_1a_3\neq 0$.
\end{example}

\bigskip

\subsection{Computing $MGC(A)$ in colength 2}

Consider a $\res$-basis $b_1,\dots,b_t$ of $I^\perp$ and an adapted $\res$-basis $\overline{F}_1,\dots,\overline{F}_{h_1},\overline{G}_1,\dots,\overline{G}_{h_2}$ of $\mathcal{L}_{A,2}$ (see \Cref{adapted}) such that
\begin{itemize}
\item $b_1,\dots,b_t,F_1,\dots,F_{h_1}$ is a $\res$-basis of $\int_\m I^\perp$,
\item $b_1,\dots,b_t,F_1,\dots,F_{h_1},G_1,\dots,G_{h_2}$ is a $\res$-basis of $\int_{\m^2} I^\perp$.
\end{itemize}

Throughout this section, we will Consider local Artin rings $A=R/I$ such that $\gcl(A)>1$. If a minimal Gorenstein cover $G=R/\ann H$ of colength 2 exists, then, by \Cref{cor}, we can assume that $H$ is a polynomial of the form
$$H=\sum_{i=1}^{h_1}\alpha_iF_i+\sum_{i=1}^{h_2}\beta_iG_i,\quad \alpha_i,\beta_i\in\res.$$
We want to obtain conditions on the $\alpha$'s and $\beta$'s under which $H$ actually generates a minimal Gorenstein cover of colength 2. By definition, $H\in\int_{\m^2}I^\perp$, hence $x_k\circ H\in\m\circ\int_\m\left(\int_\m I^\perp\right)\subseteq\int_{\m}I^\perp$ and
$$x_k\circ H=\sum_{j=1}^t\mu^j_kb_j+\sum_{j=1}^{h_1}\rho^j_kF_j,\quad \mu_k^j,\rho_k^j\in\res.$$
Set matrices $A_H=(\mu_k^j)$ and $B_H=(\rho_k^j)$. Let us describe matrix $B_H$ explicitly. We have
$$x_k\circ H=\sum_{i=1}^{h_1}\alpha_i(x_k\circ F_i)+\sum_{i=1}^{h_2}\beta_i(x_k\circ G_i).$$
Note that each $x_k\circ G_i$, for any $1\leq i\leq h_2$, is in $\int_{\m}I^\perp$ and hence it can be decomposed as
$$x_k\circ G_i=\sum_{j=1}^t\lambda^{k,i}_jb_j+\sum_{j=1}^{h_1}a^{k,i}_jF_j,\quad \lambda^{k,i}_j,a_j^{k,i}\in\res.$$
Then
$$x_k\circ H=\sum_{i=1}^{h_1}\alpha_i(x_k\circ F_i)+\sum_{i=1}^{h_2}\beta_i\left(\sum_{j=1}^t\lambda^{k,i}_jb_j+\sum_{j=1}^{h_1}a^{k,i}_jF_j\right)=b+\sum_{j=1}^{h_1}\left(\sum_{i=1}^{h_2}\beta_ia^{k,i}_j\right)F_j,$$
where $b:=\sum_{i=1}^{h_1}\alpha_i(x_k\circ F_i)+\sum_{i=1}^{h_2}\beta_i\left(\sum_{j=1}^t\lambda_j^{k,i}b_j\right)\in I^\perp$.
Observe that
\begin{equation}\label{rho}
\rho_k^j=\sum_{i=1}^{h_2}a^{k,i}_j\beta_i,
\end{equation}
hence the entries of matrix $B_H$ can be regarded as polynomials in variables $\beta_1,\dots,\beta_{h_2}$ with coefficients in $\res$.

\begin{lemma}\label{rk(B)} Consider the matrix $B_H=(\rho_k^j)$ as previously defined and let $B'_H=(\varrho^j_k)$ be the matrix of the coefficients of $\overline{L_k\circ H}=\sum_{j=1}^{h_1}\varrho_k^j\overline{F}_j\in\mathcal{L}_{A,1}$ where $L_1,\dots,L_n$ are independent linear forms. Then,
\begin{enumerate}
\item[(i)] $\rk B_H=\dim_\res\left(\displaystyle\frac{\m\circ H+I^\perp}{I^\perp}\right)$,
\item[(ii)] $\rk B'_H=\rk B_H$.
\end{enumerate}
\end{lemma}

\begin{proof}
Since $\overline{x_k\circ H}=\sum_{j=1}^{h_1}\rho_k^j\overline{F}_j$ and $\overline{F}_1,\dots,\overline{F}_{h_1}$ is a $\res$-basis of $\mathcal{L}_{A,1}$, then $\rk B_H=\dim_\res\langle\overline{x_1\circ H},\dots,\overline{x_n\circ H}\rangle_\res$. Note that $\langle\overline{x_1\circ H},\dots,\overline{x_n\circ H}\rangle_\res=(\m\circ H+I^\perp)/I^\perp\subseteq\mathcal{L}_{A,1}$, hence (i) holds.

\noindent
For (ii) it will be enough to prove that $\langle \overline{x_1\circ H},\dots,\overline{x_n\circ H}\rangle_\res=\langle \overline{L_1\circ H},\dots,\overline{L_n\circ H}\rangle_\res$.
Indeed, since $L_i=\sum_{j=1}^n\lambda^i_jx_j$ for any $1\leq i\leq n$, then $\overline{L_i\circ H}=\sum_{j=1}^n\lambda_j^i(\overline{x_j\circ H})\in\langle\overline{x_1\circ H},\dots,\overline{x_n\circ H}\rangle_\res$. The reverse inclusion comes from the fact that $L_1,\dots,L_n$ are linearly independent and hence $(L_1,\dots,L_n)=\m$.
\end{proof}

\begin{lemma}\label{BH0} With the previous notation, consider a polynomial $H\in\int_{\m^2}I^\perp$ with coefficients $\beta_1,\dots,\beta_{h_2}$ of $G_1,\dots,G_{h_2}$, respectively, and its corresponding matrix $B_H$. Then the following are equivalent:
\begin{enumerate}
\item[(i)] $B_H\neq 0$,
\item[(ii)] $\m\circ H\nsubseteq I^\perp$,
\item[(iii)] $(\beta_1,\dots,\beta_{h_2})\neq (0,\dots,0)$.
\end{enumerate}
\end{lemma}

\begin{proof}
$(i)$ implies $(ii)$. If $B_H\neq 0$, by \Cref{rk(B)}, $(\m\circ H+I^\perp)/I^\perp\neq 0$ and hence $\m\circ H\nsubseteq I^\perp$.

\noindent
$(ii)$ implies $(iii)$. If $\m\circ H\nsubseteq I^\perp$, by definition $H\notin \int_\m I^\perp$ and hence $H\in\int_{\m^2} I^\perp\backslash\int_\m I^\perp$. Therefore, some $\beta_i$ must be non-zero.

\noindent
$(iii)$ implies $(i)$. Since $G_i\in\int_{\m^2}I^\perp\backslash\int_\m I^\perp$ for any $1\leq i\leq h_2$ and, by hypothesis, there is some non-zero $\beta_i$, we have that $H\in\int_{\m^2}I^\perp\backslash \int_\m I^\perp$. We claim that $x_k\circ H\in\int_\m I^\perp\backslash I^\perp$ for some $k\in\lbrace 1,\dots,n\rbrace$. Suppose the claim is not true. Then $x_k\circ H\in I^\perp$ for any $1\leq k\leq n$, or equivalently, $\m\circ H\subseteq I^\perp$ but this is equivalent to $H\in\int_\m I^\perp$, which is a contradiction. Since
$$x_k\circ H=b+\sum_{j=1}^{h_1}\rho_k^jF_j\in\int_\m I^\perp\backslash I^\perp,\quad b\in I^\perp,$$
\noindent
for some $k\in\lbrace 1,\dots,n\rbrace$, then $\rho_k^j\neq 0$, for some $j\in\lbrace 1,\dots,h_1\rbrace$. Therefore, $B_H\neq 0$.
\end{proof}

\begin{lemma}\label{lemaB} Consider the previous setting. If $B_H=0$, then either $\gcl(A)=0$ or $\gcl(A)=1$ or $R/\ann H$ is not a cover of $A$.
\end{lemma}

\begin{proof}
If $B_H=0$, then $\m\circ H\subseteq I^\perp$ and hence $\ell(\langle H\rangle)-1\leq \ell(I^\perp)$. If $I^\perp\subseteq \langle H\rangle$, then $G=R/\ann H$ is a Gorenstein cover of $A$ such that $\ell(G)-\ell(A)\leq 1$. Therefore, either $\gcl(A)\leq 1$ or $G$ is not a cover.
\end{proof}

\bigskip
Since we already have techniques to check whether $A$ has colength 0 or 1, we can focus completely on the case $\gcl(A)>1$. Then, according to \Cref{lemaB}, if $G=R/\ann H$ is a Gorenstein cover of $A$, then $B_H\neq 0$.

\begin{proposition}\label{rk B} Assume that $B_H\neq 0$. Then $\rk B_H=1$ if and only if $(L_1,\dots,L_{n-1},L_n^2)\circ H\subseteq I^\perp$
for some independent linear forms $L_1,\dots,L_n$.
\end{proposition}

\begin{proof}
Since $B_H\neq 0$, there exists $k$ such that $x_k\circ H\notin I^\perp$. Without loss of generality, we can assume that $x_n\circ H\notin I^\perp$. If $\rk B_H=1$, then any other row of $B_H$ must be a multiple of row $n$. Therefore, for any $1\leq i\leq n-1$, there exists $\lambda_i\in\res$ such that $(x_i-\lambda_ix_n)\circ H\in I^\perp$. Take $L_n:=x_n$ and $L_i:=x_i-\lambda_ix_n$. Then $L_1,\dots,L_n$ are linearly independent and $L_i\circ H\in I^\perp$ for any $1\leq i\leq n-1$. Moreover, $L_n^2\circ H\in \m^2\circ\int_{\m^2}I^\perp\subseteq I^\perp$.

Conversely, let $B'_H=(\varrho^j_k)$ be the matrix of the coefficients of $\overline{L_k\circ H}=\sum_{j=1}^{h_1}\varrho_k^j\overline{F_j}\in\mathcal{L}_{A,1}$. By \Cref{rk(B)}, since $B_H\neq 0$, then $B'_H\neq 0$. By hypothesis, $\overline{L_1\circ H}=\dots=\overline{L_{n-1}\circ H}=0$ in $\mathcal{L}_{A,1}$ but, since $B'_H\neq 0$, then $\overline{L_n\circ H}\neq 0$. Then $\rk B'_H=1$ and hence, again by \Cref{rk(B)}, $\rk B_H=1$.
\end{proof}

Recall that $\langle H\rangle=\langle \lambda H\rangle$ for any $\lambda\in\res^\ast$. Therefore, as pointed out in \Cref{ThMGC}, for any $H\neq 0$, a Gorenstein ring $G=R/\ann H$ can be identified with a point $[H]\in\mathbb{P}_\res\left(\mathcal{L}_{A,2}\right)$ by taking coordinates $(\alpha_1:\dots:\alpha_{h_1}:\beta_1:\dots:\beta_{h_2})$. Observe that $\mathbb{P}_\res\left(\mathcal{L}_{A,2}\right)$ is a projective space over $\res$ of dimension $h_1+h_2-1$, hence we will denote it by $\mathbb{P}_\res^{h_1+h_2-1}$.

On the other hand, by \Cref{rho}, any minor of $B_H=(\rho_k^j)$ is a homogeneous polynomial in variables $\beta_1,\dots,\beta_{h_2}$. Therefore, we can consider the homogeneous ideal $\mathfrak{b}$ generated by all order-2-minors of $B_H$ in $\res[\alpha_1,\dots,\alpha_{h_1},\beta_1,\dots,\beta_{h_2}]$. Hence $\mathbb{V}_+(\mathfrak{b})$ is the projective variety consisting of all points $[H]\in\mathbb{P}_\res^{h_1+h_2-1}$ such that $\rk B_H\leq 1$.

\begin{remark} In this section we will use the notation $MGC_2(A)$ to denote the set of points $[H]\in\mathbb{P}_\res^{h_1+h_2-1}$ such that $G=R/\ann H$ is a Gorenstein cover of $A$ with $\ell(G)-\ell(A)=2$. Since we are considering rings such that $\gcl(A)>1$, we can characterize rings of higher colength than 2 as those such that $MGC_2(A)=\emptyset$. On the other hand, $\gcl(A)=2$ if and only if $MGC_2(A)\neq\emptyset$, hence in this case $MGC_2(A)=MGC(A)$, see \Cref{DefMGC} and \Cref{RemMGC}.
\end{remark}

\begin{corollary}\label{MGC1} Let $A=R/I$ be an Artin ring such that $\gcl(A)=2$.  Then
$$MGC_2(A)\subseteq\mathbb{V}_+(\mathfrak{b})\subseteq\mathbb{P}_\res^{h_1+h_2-1}.$$
\end{corollary}

\begin{proof}
By \Cref{KF}$.(ii)$, points $[H]\in MGC_2(A)$ correspond to Gorenstein covers $G=R/\ann H$ of $A$ such that $I^\perp=(L_1,\dots,L_{n-1},L_n^2)\circ H$ for some $L_1,\dots,L_n$. Since $B_H\neq 0$ by \Cref{lemaB}, then we can apply \Cref{rk B} to deduce that $\rk B_H=1$.
\end{proof}

Note that the conditions on the rank of $B_H$ do not provide any information about which particular choices of independent linear forms $L_1,\dots,L_n$ satisfy the inclusion $(L_1,\dots,L_{n-1},L_n^2)\circ H\subseteq I^\perp$. In fact, it will be enough to understand which are the $L_n$ that meet the requirements. To that end, we fix $L_n=v_1x_1+\dots+v_nx_n$, where $v=(v_1,\dots,v_n)\neq 0$. We can choose linear forms $L_i=\lambda_1^ix_1+\dots+\lambda_n^ix_n$, where $\lambda_i=(\lambda_1^i,\dots,\lambda_n^i)\neq 0$, for $1\leq i\leq n-1$, such that $L_1,\dots,L_n$ are linearly independent and $\lambda_i\cdot v=0$. Observe that the $k$-vector space generated by $L_1,\dots,L_{n-1}$ can be expressed in terms of $v_1,\dots,v_n$, that is,
$$\langle L_1,\dots,L_{n-1}\rangle_\res=\langle v_lx_k-v_kx_l:\,1\leq k<l\leq n\rangle_\res.$$

Let us now add the coefficients of $L_n$ to matrix $B_H$ by defining the following matrix depending both on $H$ and $v$:
$$C_{H,v}:=\left(\begin{array}{cccc}
\rho_1^1 & \dots & \rho_1^{h_1} & v_1\\
\vdots & & \vdots & \vdots\\
\rho_n^1 &\dots &  \rho_n^{h_1} & v_n\\
\end{array}\right).$$

\begin{proposition}\label{rk C} Assume $B_H\neq 0$. Consider $L_1,\dots,L_n$ linearly independent linear forms with $L_n=v_1x_1+\dots+v_nx_n$, $v=(v_1,\dots,v_n)\neq 0$, and $L_i=\lambda_1^ix_1+\dots+\lambda_n^ix_n$, $\lambda_i=(\lambda^i_1,\dots,\lambda^i_n)\neq 0$, such that $\lambda\cdot v=0$ for any $1\leq i\leq n-1$. Then $\rk C_{H,v}=1$ if and only if $(L_1,\dots,L_{n-1},L_n^2)\circ H\subseteq I^\perp$.
\end{proposition}

\begin{proof}
If $\rk C_{H,v}=1$, then all 2-minors of $C_{H,v}$ vanish and, in particular,
\begin{equation}\label{minors}
v_l\rho_k^j-v_k\rho_l^j=0 \mbox{ for any }1\leq k<l\leq n\mbox{ and }1\leq j\leq h_1.
\end{equation}
Recall, from \Cref{rho}, that
\begin{equation}\label{altgens}
(v_lx_k-v_kx_l)\circ H=b+\sum_{j=1}^{h_1}\left(v_l\rho_k^j-v_k\rho_l^j\right)F_j, \mbox{ where }b\in I^\perp,
\end{equation}
hence $(v_lx_k-v_kx_l)\circ H\in I^\perp$. Therefore, $L_i\circ H\in I^\perp$ for $1\leq i\leq n-1$. Moreover, $L_n^2\circ H\in \m^2\circ\int_{\m^2}I^\perp\subseteq I^\perp$.

Conversely, if $(L_1,\dots,L_{n-1},L_n^2)\circ H\subseteq I^\perp$, then $\rk B_H=1$ by \Cref{rk B}. Hence $\rk C_{H,v}=1$ if and only if \Cref{minors} holds. Since $L_i\circ H\in I^\perp$ for any $1\leq i\leq n-1$, then $(v_lx_k-v_kx_l)\circ H\in I^\perp$ and we deduce from \Cref{altgens} that \Cref{minors} is indeed satisfied.
\end{proof}

\begin{definition}\label{adm} We say that $v=(v_1,\dots,v_n)$ is an admissible vector of $H$ if $v\neq 0$ and $v_l\rho_k^j-v_k\rho_l^j=0$ for any $1\leq k<l\leq n$ and $1\leq j\leq h_1$.
\end{definition}

\begin{lemma}\label{UniqueV} Given a polynomial $H$ of the previous form such that $\rk B_H=1$:
\begin{enumerate}
\item[(i)] there always exists an admissible vector $v\in\res^n$ of $H$;
\item[(ii)] if $w\in\res^n$ such that $w=\lambda v$, with $\lambda\in\res^\ast$, then $w$ is an admissible vector of $H$;
\item[(iii)] the admissible vector of $H$ is unique up to multiplication by elements of $\res^\ast$.
\end{enumerate}
\end{lemma}

\begin{proof}
$(i)$ Since $\rk_H B=1$, \Cref{rk B} ensures the existence of linearly independent linear forms $L_1,\dots,L_n$ such that $(L_1,\dots,L_{n-1},L_n^2)\circ H\subseteq I^\perp$. By \Cref{rk C}, the vector whose components are the coefficients of $L_n$ is admissible.\\
\noindent
$(ii)$ Since $v$ is admissible, $w=\lambda v\neq 0$ and $w_l\rho_k^j-w_k\rho_l^j=\lambda(v_l\rho_k^j-v_k\rho_l^j)=0$.\\
\noindent
$(iii)$ Since $B_H\neq 0$, there exists $\rho_k^j\neq 0$ for some $1\leq j\leq h_1$ and $1\leq k\leq n$. We will first prove that $v_k\neq 0$. Suppose that $v_k=0$. By \Cref{adm}, there exists $v_i\neq 0$, $i\neq k$, and $v_i\rho_k^j-v_k\rho_i^j=0$. Then $v_i\rho_k^j=0$ and we reach a contradiction.

Consider now $w=(w_1,\dots,w_n)$ admissible with respect to $H$. From $\rho_k^jv_l-\rho_l^jv_k=0$ and $\rho_k^jw_l-\rho_l^jw_k=0$, we get $v_l=\left(\rho_l^j/\rho_k^j\right)v_k$ and $w_l=\left(\rho_l^j/\rho_k^j\right)w_k$. Set $\lambda_l:=\rho_l^j/\rho_k^j$. For any $1\leq l\leq n$, with $l\neq k$, from $v_l=\lambda_lv_k$ and $w_l=\lambda_lw_k$, we deduce that $w_l=\left(w_k/v_k\right)v_l$. Hence $w=\lambda v$, where $\lambda=w_k/v_k$, and any two admissible vectors of $H$ are linearly dependent.
\end{proof}

We now want to provide a geometric interpretation of pairs of polynomials and admissible vectors and describe the variety where they lay. Let us first note that whenever $B_H=0$, any $v\neq 0$ is an admissible vector. With this observation and \Cref{UniqueV}, for any polynomial $H$ such that $\rk B_H\leq 1$, we can consider its admissible vectors $v$ as points $[v]$ in the projective space $\mathbb{P}_\res^{n-1}$ by taking homogeneous coordinates $(v_1:\dots:v_n)$.

Let us consider the ideal generated in $\res[\alpha_1,\dots,\alpha_{h_1},\beta_1,\dots,\beta_{h_2},v_1,\dots,v_n]$ by polynomials of the form
\begin{equation}
\rho_k^j\rho_m^l-\rho_k^l\rho_m^j,\quad 1\leq k<m\leq n,1\leq j<l\leq h_1
\end{equation}
\noindent
and
\begin{equation}
v_l\rho_k^j-v_k\rho_l^j,\quad 1\leq k<l\leq n,1\leq j\leq h_1.
\end{equation}
\noindent
It can be checked that all these polynomials are bihomogeneous polynomials in the sets of variables $\alpha_1,\dots,\alpha_{h_1},\beta_1,\dots,\beta_{h_2}$ and $v_1,\dots,v_n$. Therefore, this ideal defines a variety in $\mathbb{P}_\res^{h_1+h_2-1}\times \mathbb{P}_\res^{n-1}$ the points of which satisfy the following equations:
\begin{equation}\label{Gens1}
\rho_k^j\rho_m^l-\rho_k^l\rho_m^j=0,\quad 1\leq k<m\leq n,1\leq j<l\leq h_1;
\end{equation}
\begin{equation}\label{Gens2}
v_l\rho_k^j-v_k\rho_l^j=0,\quad 1\leq k<l\leq n,1\leq j\leq h_1.
\end{equation}

\begin{definition} We denote by $\mathfrak{c}$ the ideal in $\res[\alpha_1,\dots,\alpha_{h_1},\beta_1,\dots,\beta_{h_2},v_1,\dots,v_n]$ generated by all order 2 minors of $C_{H,v}$. We denote by $\mathbb{V}_+(\mathfrak{c})$ the variety defined by $\mathfrak{c}$ in $\mathbb{P}_\res^{h_1+h_2-1}\times \mathbb{P}_\res^{n-1}$.
\end{definition}

\begin{lemma}\label{vc} With the previous definitions, the set of points of $\mathbb{V}_+(\mathfrak{c})$ is
$$\left\lbrace ([H],[v])\in \mathbb{P}_\res^{h_1+h_2-1}\times \mathbb{P}_\res^{n-1}\mid [H]\in\mathbb{V}_+(\mathfrak{b})\mbox{ and $v$ admissible with respect to $H$}\right\rbrace.$$
\end{lemma}

\begin{proof}
It follows from \Cref{Gens1} and \Cref{Gens2}.
\end{proof}

\begin{lemma}\label{proj}
Let $\pi_1$ be the projection map from $\mathbb{P}_\res^{h_1+h_2-1}\times \mathbb{P}_\res^{n-1}\longrightarrow \mathbb{P}_\res^{h_1+h_2-1}$. Then $\pi_1(\mathbb{V}_+(\mathfrak{c}))=\mathbb{V}_+(\mathfrak{b})$. Moreover, $\pi_1$ is a bijection over the subset of $\mathbb{V}_+(\mathfrak{c})$ where $\rk B_H=1$.
\end{lemma}

\begin{proof}
Any element of $\mathbb{V}_+(\mathfrak{c})$ is of the form $([H],[v])$ described in \Cref{vc}. Then $\pi_1([H],[v])=[H]\in\mathbb{V}_+(\mathfrak{b})$. Conversely, given an element $[H]\in\mathbb{V}_+(\mathfrak{b})$, then $\rk B_H\leq 1$. If $B_H=0$, then any $v\neq 0$ satisfies $([H],[v])\in\mathbb{V}_+(\mathfrak{c})$. If $\rk B=1$, by \Cref{UniqueV}, there exist a unique admissible $v$ up to scalar multiplication, hence $([H],[v])$ is the unique point in $\mathbb{V}_+(\mathfrak{c})$ such that $\pi_1([H],[v])=[H]$.
\end{proof}

From \Cref{MGC1}, we know that all covers $G=R/\ann H$ of of $A=R/I$ colength 2 correspond to points $[H]\in\mathbb{V}_+(\mathfrak{b})$ but, in general, not all points of $\mathbb{V}_+(\mathfrak{b})$ correspond to such covers. Therefore, we need to identify and remove those $[H]$ such that $(L_1,\dots,L_{n-1},L_n^2)\circ H\subsetneq I^\perp$.\\

As $\res$-vector space, $(L_1,\dots,L_{n-1},L_n^2)\circ H$ is generated by
\begin{itemize}
\item $(v_l x_k-v_k x_l)\circ H$, $1\leq k<l\leq n$;
\item $x^\theta\circ H$, $2\leq\vert\theta\vert\leq s+2$.
\end{itemize}
\noindent
Since $(L_1,\dots,L_{n-1},L_n^2)\circ H\subseteq I^\perp$, we can provide an explicit description of these generators with respect to  the $\res$-basis $b_1,\dots,b_t$ of $I^\perp$ as follows:
$$(x_kv_l-x_lv_k)\circ H=\sum_{j=1}^t\left(v_l\sum_{i=1}^{h_1}\alpha_i\mu_j^{k,i}-v_k\sum_{i=1}^{h_1}\alpha_i\mu_j^{l,i}+v_l\sum_{i=1}^{h_2}\beta_i\lambda_j^{k,i}-v_k\sum_{i=1}^{h_2}\beta_i\lambda_j^{l,i}\right)b_j,$$

\noindent
for $1\leq l<k\leq n$, with $x_k\circ F_i=\sum_{j=1}^t\mu_j^{k,i}b_j$ and
$x_k\circ G_i=\sum_{j=1}^t\lambda_j^{k,i}b_j+\sum_{j=1}^{h_1}a_j^{k,i}F_j$, $\mu_j^{k,i},\lambda_j^{k,i},a_j^{k,i}\in\res$;
$$x^{\theta}\circ H=\sum_{j=1}^t\left(\sum_{i=1}^{h_1}\mu_j^{\theta,i}\alpha_i+\sum_{i=1}^{h_2}\lambda_j^{\theta,i}\beta_i\right)b_j,$$

\noindent
where $2\leq\vert\theta\vert\leq s+2$, $x^{\theta}\circ F_i=\sum_{j=1}^t\mu_j^{\theta,i}b_j$ and
$x^{\theta}\circ G_i=\sum_{j=1}^t\lambda_j^{\theta,i}b_j$, $\mu_j^{\theta,i},\lambda_j^{\theta,i}\in\res$.

We now define matrix $U_{H,v}$ such that its rows are the coefficients of each generator of $(L_1,\dots,L_{n-1},L_n^2)\circ H$ with respect to the $\res$-basis $b_1,\dots,b_t$ of $I^\perp$:
$$\begin{array}{r|ccc}
• & b_1 & \dots & b_t\\
\hline
(x_2v_1-x_1v_2)\circ  H & \varrho^1_{1,2} & \cdots & \varrho^t_{1,2}\\
\vdots & \vdots &  & \vdots \\
(x_nv_{n-1}-x_{n-1}v_n)\circ H & \varrho^1_{n-1,n} & \cdots & \varrho^t_{n-1,n} \\
x_1^2\circ H  & \varsigma^1_{(2,0,\dots,0)} & \cdots & \varsigma^t_{(2,0,\dots,0)} \\
x_1x_2\circ H  & \varsigma^1_{(1,1,0,\dots,0)} &  \cdots & \varsigma^t_{(1,1,0,\dots,0)} \\
\vdots  & \vdots &  & \vdots \\
x_n^2\circ H  & \varsigma^1_{(0,\dots,0,2)} & \cdots & \varsigma^t_{(0,\dots,0,2)} \\
\vdots  & \vdots &  & \vdots \\
x_n^{s+2}\circ H  & \varsigma^1_{(0,\dots,0,s+2)} & \cdots &  \varsigma^t_{(0,\dots,0,s+2)}\\
\end{array}$$
\noindent
where
$$\varrho^j_{l,k}:=v_l\sum_{i=1}^{h_1}\alpha_i\mu_j^{k,i}-v_k\sum_{i=1}^{h_1}\alpha_i\mu_j^{l,i}+v_l\sum_{i=1}^{h_2}\beta_i\lambda_j^{k,i}-v_k\sum_{i=1}^{h_2}\beta_i\lambda_j^{l,i}$$
\noindent
and
$$\varsigma^j_{\theta}:=\sum_{i=1}^{h_1}\mu_j^{\theta,i}\alpha_i+\sum_{i=1}^{h_2}\lambda_j^{\theta,i}\beta_i.$$
\noindent
It can be easily checked that the entries of this matrix are either bihomogeneous polynomials $\varrho^j_{l,k}$ in variables $((\alpha,\beta),v)$ of bidegree $(1,1)$ or homogeneous polynomials $\varsigma^j_{\theta}$ in variables $(\alpha,\beta)$ of degree 1. Let $\mathfrak{a}$ be the ideal in $\res[\alpha_1,\dots,\alpha_{h_1},\beta_1,\dots,\beta_{h_2},v_1,\dots,v_n]$ generated by all minors of $U_{H,v}$ of order $t=\dim_\res I^\perp$. It can be checked that $\mathfrak{a}$ is a bihomogeneous ideal in variables $((\alpha,\beta),v)$, hence we can think of $\mathbb{V}_+(\mathfrak{a})$ as the following variety in $\mathbb{P}^{h_1+h_2-1}\times\mathbb{P}^{n-1}$:
$$\mathbb{V}_+(\mathfrak{a})=\lbrace ([H],[v])\in\mathbb{P}^{h_1+h_2-1}\times\mathbb{P}^{n-1}\mid \rk U_{H,v}<t\rbrace.$$

\begin{proposition}\label{MGC2} Assume $\gcl(A)>1$. Consider a point $([H],[v])\in\mathbb{V}_+(\mathfrak{c})\subset\mathbb{P}^{h_1+h_2-1}\times\mathbb{P}^{n-1}$. Then
$$[H]\in MGC_2(A)\Longleftrightarrow ([H],[v])\notin \mathbb{V}_+(\mathfrak{a}),$$
\end{proposition}

\begin{proof}
From \Cref{MGC1} we deduce that if $[H]$ is a point in $MGC_2(A)$, then $\rk B_H\leq 1$. The same is true for any point $([H],[v])\in\mathbb{V}_+(\mathfrak{c})$. Let us consider these two cases:

\emph{Case $B_H=0$}. Since $\gcl(A)>1$, then $R/\ann H$ is not a Gorenstein cover of $A$ by \Cref{lemaB}, hence $[H]\notin MGC_2(A)$. On the other hand, as stated in the proof of \Cref{proj}, $([H],[v])\in\mathbb{V}_+(\mathfrak{c})$ for any $v\neq 0$. By \Cref{BH0} and $\gcl(A)\neq 1$, it follows that
$$(L_1,\dots,L_{n-1},L_n^2)\circ H\subseteq \m\circ H\subsetneq I^\perp$$
\noindent
for any $L_1,\dots,L_n$ linearly independent linear forms, where $L_n=v_1x_1+\dots+v_nx_n$.   Therefore, the rank of matrix $U_{H,v}$ is always strictly smaller than $\dim_\res I^\perp$. Hence $([H],[v])\in\mathbb{V}_+(\mathfrak{a})$ for any $v\neq 0$.

\emph{Case $\rk B_H=1$}. If $[H]\in MGC_2(A)$, then there exist $L_1,\dots,L_n$ such that $(L_1,\dots,L_{n-1},L_n^2)\circ H=I^\perp$. Take $v$ as the vector of coefficients of $L_n$, it is an admissible vector by definition. By \Cref{proj}, $([H],[v])\in\mathbb{V}_+(\mathfrak{c})$ is unique and $\rk U_{H,v}=\dim_\res I^\perp$. Therefore, $([H],[v])\notin\mathbb{V}_+(\mathfrak{a})$.

Conversely, if $([H],[v])\in\mathbb{V}_+(\mathfrak{c})\cap\mathbb{V}_+(\mathfrak{a})$, then $\rk U_{H,v}<\dim_\res I^\perp$ and hence $(L_1,\dots,L_{n-1},L_n^2)\circ H\subsetneq I^\perp$, where $L_n=v_1x_1+\dots+v_nx_n$. By unicity of $v$, no other choice of $L_1,\dots,L_n$ satisfies the inclusion $(L_1,\dots,L_{n-1},L_n^2)\circ H\subset I^\perp$, hence $[H]\notin MGC_2(A)$.
\end{proof}

\begin{corollary} Assume $\gcl(A)>1$. With the previous definitions,
$$MGC_2(A)=\mathbb{V}_+(\mathfrak{b})\backslash \pi_1\left(\mathbb{V}_+(\mathfrak{c})\cap\mathbb{V}_+(\mathfrak{a})\right).$$
\end{corollary}

\begin{proof}
It follows from \Cref{proj} and \Cref{MGC2}.
\end{proof}

Finally, let us recall the following result for bihomogeneous ideals:

\begin{lemma} Let ideals $\mathfrak{a},\mathfrak{c}$ be as previously defined, $\mathfrak{d}=\mathfrak{a}+\mathfrak{c}$ the sum ideal and $\pi_1:\mathbb{P}_\res^{h_1+h_2-1}\times \mathbb{P}_\res^{n-1}\longrightarrow \mathbb{P}_\res^{h_1+h_2-1}$ be the projection map. Let $\widehat{\mathfrak{d}}$ be the projective elimination of the ideal $\mathfrak{d}$ with respect to variables $v_1,\dots,v_n$. Then,
$$\pi_1(\mathbb{V}_+(\mathbb{\mathfrak{a}})\cap\mathbb{V}_+(\mathbb{\mathfrak{c}}))=\mathbb{V}_+(\widehat{\mathfrak{d}}).$$
\end{lemma}
\begin{proof}
See \cite[Section 8.5, Exercise 16]{CLO97}.
\end{proof}

\bigskip

We end this section by providing an algorithm to effectively compute the set $MGC_2(A)$ of any ring $A=R/I$ such that $\gcl(A)>1$.

\begin{algorithm}[H]
\caption{Compute $MGC_2(A)$ of $A=R/I$ with $n\geq 2$ and $\gcl(A)>1$}
\label{AlMGC2}
\begin{algorithmic}
\REQUIRE
$s$ socle degree of $A=R/I$;
$b_1,\dots,b_t$ $\res$-basis of the inverse system $I^\perp$;
$F_1,\dots,F_{h_1},G_1,\dots,G_{h_2}$ an adapted $\res$-basis of $\mathcal{L}_{A,2}$;
$U_1,\dots,U_n$ contraction matrices of $\int_{\m^2} I^\perp$.
\ENSURE
ideals $\mathfrak{b}$ and $\widehat{\mathfrak{d}}$ such that
$MGC_2(A)=\mathbb{V}_+(\mathfrak{b})\backslash\mathbb{V}_+(\widehat{\mathfrak{d}})$.
\RETURN
\begin{enumerate}
\item Set $H=\alpha_1F_1+\dots+\alpha_{h_1}F_{h_1}+\beta_1G_1+\dots+\beta_{h_2}G_{h_2}$, where $\alpha,\beta$ are variables in $\res$. Set column vectors $\textbf{H}=(0,\dots,0,\alpha,\beta)^t$ and $v=(v_1,\dots,v_n)^t$ in $R=\res[\alpha,\beta,v]$, where the first $t$ components of $\textbf{H}$ are zero.
\item Build matrix $B_H=(\rho_i^j)_{1\leq i\leq n,\,1\leq j\leq h_1}$, where
$U_i\textbf{H}$ is the column vector $(\mu_i^1,\dots,\mu_i^t,\rho_i^1,\dots,\rho_i^{h_1},0,\dots,0)^t$.
\item Build matrix $C_{H,v}=\left(\begin{array}{c|c}
B_H & v\\
\end{array}\right)$ as an horizontal concatenation of $B_H$ and the column vector $v$.
\item Compute the ideal $\mathfrak{c}\subseteq R$ generated by all minors of order 2 of $B_H$.
\item Build matrix $U_{H,v}$ as a vertical concatenation of matrices $(\varrho^j_{l,k})_{1\leq j\leq h_1,\,1\leq l<k\leq n}$ and $(\varsigma_\theta^j)_{2\leq\vert\theta\vert\leq s+2,\,1\leq j\leq h_1}$, such that $(v_l U_k-v_k U_l)\textbf{H}=(\varrho_{l,k}^1,\cdots,\varrho_{l,k}^{h_1},0,\cdots,0)^t$ and $U^\theta \textbf{H}=(\varsigma_\theta^1,\cdots,\varsigma_\theta^{h_1},0,\cdots,0)^t$, with $1\leq k<l\leq n$ and $2\leq\vert\theta\vert\leq s+2$.
\item Compute the ideal $\mathfrak{a}\subseteq R$ generated by all minors of order $t$ of $U_{H,v}$ and the ideal $\mathfrak{d}=\mathfrak{a}+\mathfrak{c}\subseteq R$ .
\item Compute $\widehat{\mathfrak{d}}\subseteq R'=\res[\alpha,\beta]$, where $\widehat{\cdot}$ denotes the projective elimination of the ideal in $R$ with respect to variables $v_1,\dots,v_n$.
\item Compute the ideal $\mathfrak{b}:=\widehat{\mathfrak{c}}\subseteq R'$.
\end{enumerate}
\end{algorithmic}
\end{algorithm}

The output of \Cref{AlMGC2} can be interpreted as $MGC_2(A)=\mathbb{V}_+(\mathfrak{b})\backslash \mathbb{V}_+(\mathfrak{\widehat{d}})$. Moreover, any point $[\alpha_1:\dots:\alpha_{h_1}:\beta_1:\dots:\beta_{h_2}]\in MGC_2(A)$ corresponds to a minimal Gorenstein cover $G=R/\ann H$ of colength 2 of $A$, where $H=\alpha_1F_1+\dots+\alpha_{h_1}F_{h_1}+\beta_1G_1+\dots+\beta_{h_2}G_{h_2}$. If $MGC_2(A)\neq\emptyset$, then $\gcl(A)=2$ and hence $MGC(A)=MGC_2(A)$. Otherwise, $\gcl(A)>2$.

\begin{example}
Consider $A=R/I$, with $R=\res[\![x_1,x_2]\!]$ and $I=(x_1^2,x_1x_2^2,x_2^4)$. Applying \Cref{AlINT} twice we get the necessary input for \Cref{AlMGC2}:\\
{\sc Input}: $b_1=1,b_2=y_1,b_3=y_2,b_4=y_2^2,b_5=y_1y_2,b_6=y_2^3$ $\res$-basis of $I^\perp$; $F_1=y_2^4,F_2=y_1y_2^2,F_3=y_1^2,G_1=y_1^2y_2,G_2=y_1y_2^3,G_3=y_2^5,G_4=y_1^3$ adapted $\res$-basis of $\mathcal{L}_{A,2}$; $U_1, U_2$ contraction matrices of $\int_{\m^2}I^\perp$.\\
{\sc Output}: $\mathfrak{b}=(b_3b_4,b_2b_4)$, $\mathfrak{\widehat{d}}=(b_3b_4,b_2b_4,b_2^2-b_1b_3)$.\\
$MGC_2(A)=\mathbb{V}_+(b_3b_4,b_2b_4)\backslash\mathbb{V}_+ (b_3b_4,b_2b_4,b_2^2-b_1b_3)=\mathbb{V}_+(b_3b_4,b_2b_4)\backslash\mathbb{V}_+ (b_2^2-b_1b_3).$
Note that if $b_3b_4=b_2b_4=0$ and $b_4\neq 0$, then both $b_2$ and $b_3$ are zero and the condition $b_2^2-b_1b_3=0$ always holds. Therefore, $\gcl(A)=2$ and hence
$$MGC(A)=\mathbb{V}_+(b_4)\backslash\mathbb{V}_+ (b_2^2-b_1b_3)\simeq \mathbb{P}^5\backslash\mathbb{V}_+ (b_2^2-b_1b_3),$$
\noindent
where $(a_1:a_2:a_3:b_1:b_2:b_3)$ are the coordinates of the points in $\mathbb{P}^5$. Moreover, any minimal Gorenstein cover is of the form $G=R/\ann H$, where
$$H=a_1y_2^4+a_2y_1y_2^2+a_3y_1^2+b_1y_1^2y_2+b_2y_1y_2^3+b_3y_2^5$$
\noindent
satisfies $b_2^2-b_1b_3\neq 0$. All such covers admit $(x_1,x_2^2)$ as the corresponding $K_H$.
\end{example}

\section{Computations}

The first aim of this section is to provide a wide range of examples of the computation of the minimal Gorenstein cover variety of a local ring $A$. In \cite{Poo08a}, Poonen provides a complete classification of local algebras over an algebraically closed field of length equal or less than 6. Note that, for higher lengths, the number of isomorphism classes is no longer finite. We will go through all algebras of Poonen's list and restrict, for the sake of simplicity, to fields of characteristic zero.

On the other hand, we also intend to test the efficiency of the algorithms by collecting the computation times. We have implemented algorithms 1, 2 and 3 of \Cref{s5} in the commutative algebra software \emph{Singular} \cite{DGPS}. The computer we use runs into the operating system Microsoft Windows 10 Pro and its technical specifications are the following: Surface Pro 3; Processor: 1.90 GHz Intel Core i5-4300U 3 MB SmartCache; Memory: 4GB 1600MHz DDR3.

\subsection{Teter varieties}

In this first part of the section we are interested in the computation of Teter varieties, that is, the $MGC(A)$ variety for local $\res$-algebras $A$ of Gorenstein colength 1. All the results are obtained by running \Cref{AlMGC1} in \emph{Singular}.

\begin{example}
Consider $A=R/I$, with $R=\res[\![x_1,x_2,x_3]\!]$ and $I=(x_1^2,x_1x_2,x_1x_3,x_2x_3,x_2^3,x_3^3)$. Note that $\HF_A=\lbrace 1,3,2\rbrace$ and $\tau(A)=3$. The output provided by our implementation of the algorithm in \emph{Singular} \cite{DGPS} is the following:
{\scriptsize{
\begin{verbatim}
F;
a(4)*x(2)^3+a(1)*x(3)^3+a(6)*x(1)^2+a(5)*x(1)*x(2)+a(3)*x(1)*x(3)+a(2)*x(2)*x(3)
radical(a);
a(1)*a(4)*a(6)
\end{verbatim}}}
\noindent
We consider points with coordinates $(a_1:a_2:a_3:a_4:a_5:a_6)\in\mathbb{P}^5$. Therefore, $MGC(A)=\mathbb{P}^5\backslash\mathbb{V}_+(a_1a_4a_6)$ and any minimal Gorenstein cover is of the form $G=R/\ann H$, where $H=a_1y_3^3+a_2y_2y_3+a_3y_1y_3+a_4y_2^3+a_5y_1y_2+a_6y_1^2$ with $a_1a_4a_6\neq 0$.
\end{example}

In \Cref{table:TMGC1} below we show the computation time (in seconds) of all isomorphism classes of local $\res$-algebras $A$ of $\gcl(A)=1$ appearing in Poonen's classification \cite{Poo08a}. In this table, we list the Hilbert function of $A=R/I$, the expression of the ideal $I$ up to linear isomorphism, the dimension $h-1$ of the projective space $\mathbb{P}^{h-1}$ where the variety $MGC(A)$ lies and the computation time. Note that our implementation of \Cref{AlMGC1} includes also the computation of the $\res$-basis of $\int_\m I^\perp$, hence the computation time corresponds to the total.

Note that \Cref{AlMGC1} also allows us to prove that all the other non-Gorenstein local rings appearing in Poonen's list have Gorenstein colength at least 2.

{\footnotesize{
\begin{table}[H]
\begin{tabular}{|c|c|c|c|}
\hline
HF(R/I) & I & h-1 & t(s)\\
\hline
$1,2$ & $(x_1,x_2)^2$ & 2 & 0,06\\
$1,2,1$ &  $x_1x_2,x_2^2,x_1^3$ & 2 & 0,06\\
$1,3 $ &  $(x_1,x_2,x_3)^2$ & 5 & 0,13\\
$1,2,1,1 $ &  $x_1^2,x_1x_2,x_2^4$ & 2 & 0,23\\
$1,2,2$ &  $x_1x_2,x_1^3,x_2^3$ & 2 & 0,11\\
 & $x_1x_2^2,x_1^2,x_2^3$ & 2 & 0,05\\
$1,3,1$ & $x_1x_2,x_1x_3,x_2x_3,x_2^2,x_3^2,x_1^3$ & 5 & 0,16\\
$1,4$ & $(x_1,x_2,x_3,x_4)^2$ & 9 & 2,30\\
$1,2,1,1,1$ & $x_1x_2,x_1^5,x_2^2$ & 2 & 0,17\\
$1,2,2,1$ & $x_1x_2,x_1^3,x_2^4$ & 2 & 0,09\\
 & $x_1^2+x_2^3,x_1x_2^2,x_2^4$ & 2 & 0,1\\
$1,3,1,1$ & $x_1x_2,x_1x_3,x_2x_3,x_2^2,x_3^2,x_1^4$ & 5& 3,05\\
$1,3,2$ & $x_1^2,x_1x_2,x_1x_3,x_2^2,x_2x_3^2,x_3^3$ & 5 & 0,33\\
 & $x_1^2,x_1x_2,x_1x_3,x_2x_3,x_2^3,x_3^3$ & 5 & 0,23\\
$1,4,1$ & $x_1x_2,x_1x_3,x_1x_4,x_2x_3,x_2x_4,x_3x_4,x_2^2,x_3^2,x_4^2,x_1^3$ & 9 & 3,21\\
$1,5$ & $(x_1,x_2,x_3,x_4,x_5)^2$ & 14 & 1,25\\
\hline
\end{tabular}
\medskip
\caption[MGC1]{Computation times of $MGC(A)$ of local rings $A=R/I$ with $\ell(A)\leq 6$ and $\gcl(A)=1$.}
\label{table:TMGC1}
\end{table}}}

\subsection{Minimal Gorenstein covers variety in colength 2}
Now we want to compute $MGC(A)$ for $\gcl(A)=2$. All the examples are obtained by running \Cref{AlMGC2} in \emph{Singular}.
\begin{example}
Consider $A=R/I$, with $R=\res[\![x_1,x_2,x_3]\!]$ and $I=(x_1^2,x_2^2,x_3^2,x_1x_2,x_1x_3)$. Note that $\HF_A=\lbrace 1,3,1\rbrace$ and $\tau(A)=2$. The output provided by our implementation of the algorithm in \emph{Singular} \cite{DGPS} is the following:
{\tiny{
\begin{multicols}{2}
\begin{verbatim}
H;
b(10)*x(1)^3+b(7)*x(1)^2*x(2)+
+b(8)*x(1)*x(2)^2+b(9)*x(2)^3+
+b(1)*x(1)^2*x(3)+b(2)*x(1)*x(2)*x(3)+
+b(3)*x(2)^2*x(3)+b(4)*x(1)*x(3)^2+
+b(6)*x(2)*x(3)^2+b(5)*x(3)^3+
+a(5)*x(1)^2+a(4)*x(1)*x(2)+
+a(3)*x(2)^2+a(2)*x(1)*x(3)+
+a(1)*x(3)^2
radical(b);
_[1]=b(8)^2-b(7)*b(9)
_[2]=b(7)*b(8)-b(9)*b(10)
_[3]=b(6)*b(8)-b(4)*b(9)
_[4]=b(3)*b(8)-b(2)*b(9)
_[5]=b(2)*b(8)-b(1)*b(9)
_[6]=b(1)*b(8)-b(3)*b(10)
_[7]=b(7)^2-b(8)*b(10)
_[8]=b(6)*b(7)-b(4)*b(8)
_[9]=b(4)*b(7)-b(6)*b(10)
_[10]=b(3)*b(7)-b(1)*b(9)
_[11]=b(2)*b(7)-b(3)*b(10)
_[12]=b(1)*b(7)-b(2)*b(10)
_[13]=b(3)*b(6)-b(5)*b(9)
_[14]=b(2)*b(6)-b(5)*b(8)
_[15]=b(1)*b(6)-b(5)*b(7)
_[16]=b(2)*b(5)-b(4)*b(6)
_[17]=b(4)^2-b(1)*b(5)
_[18]=b(3)*b(4)-b(5)*b(8)
_[19]=b(2)*b(4)-b(5)*b(7)
_[20]=b(1)*b(4)-b(5)*b(10)
_[21]=b(2)*b(3)-b(4)*b(9)
_[22]=b(1)*b(3)-b(4)*b(8)
_[23]=b(2)^2-b(4)*b(8)
_[24]=b(1)*b(2)-b(6)*b(10)
_[25]=b(1)^2-b(4)*b(10)
_[26]=b(3)*b(5)*b(10)-b(6)^2*b(10)
_[27]=b(3)^2*b(10)-b(6)*b(9)*b(10)
_[28]=b(4)*b(6)^2-b(5)^2*b(8)
_[29]=b(6)^3*b(10)-b(5)^2*b(9)*b(10)
radical(d);
_[1]=b(8)^2-b(7)*b(9)
_[2]=b(7)*b(8)-b(9)*b(10)
_[3]=b(6)*b(8)-b(4)*b(9)
_[4]=b(3)*b(8)-b(2)*b(9)
_[5]=b(2)*b(8)-b(1)*b(9)
_[6]=b(1)*b(8)-b(3)*b(10)
_[7]=b(7)^2-b(8)*b(10)
_[8]=b(6)*b(7)-b(4)*b(8)
_[9]=b(4)*b(7)-b(6)*b(10)
_[10]=b(3)*b(7)-b(1)*b(9)
_[11]=b(2)*b(7)-b(3)*b(10)
_[12]=b(1)*b(7)-b(2)*b(10)
_[13]=b(3)*b(6)-b(5)*b(9)
_[14]=b(2)*b(6)-b(5)*b(8)
_[15]=b(1)*b(6)-b(5)*b(7)
_[16]=b(2)*b(5)-b(4)*b(6)
_[17]=b(4)^2-b(1)*b(5)
_[18]=b(3)*b(4)-b(5)*b(8)
_[19]=b(2)*b(4)-b(5)*b(7)
_[20]=b(1)*b(4)-b(5)*b(10)
_[21]=b(2)*b(3)-b(4)*b(9)
_[22]=b(1)*b(3)-b(4)*b(8)
_[23]=b(2)^2-b(4)*b(8)
_[24]=b(1)*b(2)-b(6)*b(10)
_[25]=b(1)^2-b(4)*b(10)
_[26]=b(3)*b(5)*b(10)-b(6)^2*b(10)
_[27]=b(3)^2*b(10)-b(6)*b(9)*b(10)
_[28]=b(4)*b(6)^2-b(5)^2*b(8)
_[29]=a(5)*b(3)*b(5)-a(5)*b(6)^2
_[30]=a(5)*b(3)^2-a(5)*b(6)*b(9)
_[31]=b(6)^3*b(10)-b(5)^2*b(9)*b(10)
_[32]=a(5)*b(6)^3-a(5)*b(5)^2*b(9)
\end{verbatim}
\end{multicols}}}
\noindent
We can simplify the output by using the primary decomposition of the ideal $\mathfrak{b}=\bigcap_{i=1}^k \mathfrak{b}_i$. Then,
$$MGC(A)=\left(\bigcup_{i=1}^k \mathbb{V}_+(\mathfrak{b}_i)\right)\backslash\mathbb{V}_+(\mathfrak{\widehat{d}})=\bigcup_{i=1}^k \left(\mathbb{V}_+(\mathfrak{b}_i)\backslash\mathbb{V}_+(\mathfrak{\widehat{d}})\right).$$
\noindent
\emph{Singular} \cite{DGPS} provides a primary decomposition $\mathfrak{b}=\mathfrak{b}_1\cap \mathfrak{b}_2$ that satisfies $\mathbb{V}_+(\mathfrak{b}_2)\backslash\mathbb{V}_+(\mathfrak{\widehat{d}})=\emptyset$. Therefore, we get
$$MGC(A)=\mathbb{V}_+(b_1,b_2,b_4,b_7,b_8,b_{10},b_3b_6-b_5b_9)\backslash\left(\mathbb{V}_+(a_5)\cup\mathbb{V}_+ (-b_6^3+b_5^2b_9,b_3b_5-b_6^2,b_3^2-b_6b_9)\right).$$
\noindent
in $\mathbb{P}^{14}$. We can eliminate some of the variables and consider $MGC(A)$ to be the following variety:
$$MGC(A)=\mathbb{V}_+(b_3b_6-b_5b_9)\backslash\left(\mathbb{V}_+(a_5)\cup\mathbb{V}_+ (b_5^2b_9-b_6^3,b_3b_5-b_6^2,b_3^2-b_6b_9)\right)\subset \mathbb{P}^8.$$
\noindent
Therefore, any minimal Gorenstein cover is of the form $G=R/\ann H$, where
$$H=a_1y_3^2+a_2y_1y_3+a_3y_2^2+a_4y_1y_2+a_5y_1^2+b_3y_2^2y_3+b_5y_3^3+b_6y_2y_3^2+b_9y_2^3$$
\noindent
satisfies $b_3b_6-b_5b_9=0$, $a_5\neq 0$ and at least one of the following conditions: $b_5^2b_9-b_6^3\neq 0, b_3b_5-b_6^2\neq 0, b_3^2-b_6b_9\neq 0$.

\noindent
Moreover, note that $\mathbb{V}_+(\mathfrak{c})\backslash\mathbb{V}_+(\mathfrak{a})=\mathbb{V}_+(\mathfrak{c_1})\backslash\mathbb{V}_+(\mathfrak{a})$, where $\mathfrak{c}=\mathfrak{c}_1\cap \mathfrak{c}_2$ is the primary decomposition of $\mathfrak{c}$ and $\mathfrak{c}_1=\mathfrak{b}_1+(v_1,v_2b_5-v_3b_6,v_2b_3-v_3b_9)$. Hence, any $K_H$ such that $K_H\circ H=I^\perp$ will be of the form $K_H=(L_1,L_2,L_3^2)$, where $L_1,L_2,L_3$ are independent linear forms in $R$ such that $L_3=v_2x_2+v_3x_3$, with $v_2b_5-v_3b_6=v_2b_3-v_3b_9=0$.
\end{example}

\begin{example}
Consider $A=R/I$, with $R=\res[\![x_1,x_2,x_3]\!]$ and $I=(x_1x_2,x_1x_3,x_2x_3,x_2^2,x_3^2-x_1^3)$. Note that $\HF_A=\lbrace 1,3,1,1\rbrace$ and $\tau(A)=2$. The output provided by our implementation of the algorithm in \emph{Singular} \cite{DGPS} is the following:
{\tiny{
\begin{multicols}{2}
\begin{verbatim}
H;
-b(10)*x(1)^4+b(9)*x(1)^2*x(2)+
+b(7)*x(1)*x(2)^2+b(8)*x(2)^3+
+b(6)*x(1)^2*x(3)+b(1)*x(1)*x(2)*x(3)+
+b(2)*x(2)^2*x(3+b(3)*x(1)*x(3)^2+
+b(4)*x(2)*x(3)^2+b(5)*x(3)^3+
+a(5)*x(1)*x(2)+a(4)*x(2)^2+
+a(3)*x(1)*x(3)+a(2)*x(2)*x(3)+
+a(1)*x(3)^2
radical(b);
_[1]=b(8)*b(10)
_[2]=b(7)*b(10)
_[3]=b(4)*b(10)
_[4]=b(2)*b(10)
_[5]=b(1)*b(10)
_[6]=b(6)*b(8)-b(2)*b(9)
_[7]=b(7)^2-b(8)*b(9)
_[8]=b(6)*b(7)-b(1)*b(9)
_[9]=b(4)*b(7)-b(3)*b(8)
_[10]=b(3)*b(7)-b(4)*b(9)
_[11]=b(2)*b(7)-b(1)*b(8)
_[12]=b(1)*b(7)-b(2)*b(9)
_[13]=b(4)*b(6)-b(5)*b(9)
_[14]=b(2)*b(6)-b(4)*b(9)
_[15]=b(1)*b(6)-b(3)*b(9)
_[16]=b(4)^2-b(2)*b(5)
_[17]=b(3)*b(4)-b(1)*b(5)
_[18]=b(2)*b(4)-b(5)*b(8)
_[19]=b(1)*b(4)-b(5)*b(7)
_[20]=b(3)^2-b(5)*b(6)+b(3)*b(10)
_[21]=b(2)*b(3)-b(5)*b(7)
_[22]=b(1)*b(3)-b(5)*b(9)
_[23]=b(2)^2-b(4)*b(8)
_[24]=b(1)*b(2)-b(3)*b(8)
_[25]=b(1)^2-b(4)*b(9)
_[26]=b(5)*b(9)*b(10)
_[27]=b(3)*b(9)*b(10)
radical(d);
_[1]=b(8)*b(10)
_[2]=b(7)*b(10)
_[3]=b(4)*b(10)
_[4]=b(2)*b(10)
_[5]=b(1)*b(10)
_[6]=b(6)*b(8)-b(2)*b(9)
_[7]=b(7)^2-b(8)*b(9)
_[8]=b(6)*b(7)-b(1)*b(9)
_[9]=b(4)*b(7)-b(3)*b(8)
_[10]=b(3)*b(7)-b(4)*b(9)
_[11]=b(2)*b(7)-b(1)*b(8)
_[12]=b(1)*b(7)-b(2)*b(9)
_[13]=b(4)*b(6)-b(5)*b(9)
_[14]=b(2)*b(6)-b(4)*b(9)
_[15]=b(1)*b(6)-b(3)*b(9)
_[16]=b(4)^2-b(2)*b(5)
_[17]=b(3)*b(4)-b(1)*b(5)
_[18]=b(2)*b(4)-b(5)*b(8)
_[19]=b(1)*b(4)-b(5)*b(7)
_[20]=b(3)^2-b(5)*b(6)+b(3)*b(10)
_[21]=b(2)*b(3)-b(5)*b(7)
_[22]=b(1)*b(3)-b(5)*b(9)
_[23]=b(2)^2-b(4)*b(8)
_[24]=b(1)*b(2)-b(3)*b(8)
_[25]=b(1)^2-b(4)*b(9)
_[26]=b(5)*b(9)*b(10)
_[27]=b(3)*b(9)*b(10)
_[28]=a(4)*b(5)*b(10)
_[29]=a(4)*b(3)*b(10)
\end{verbatim}
\end{multicols}}}
\noindent
\emph{Singular} \cite{DGPS} provides a primary decomposition $\mathfrak{b}=\mathfrak{b}_1\cap \mathfrak{b}_2\cap\mathfrak{b}_3$ such that $\mathbb{V}_+(\mathfrak{b})\backslash\mathbb{V}_+(\mathfrak{\widehat{d}})=\mathbb{V}_+(\mathfrak{b}_2)\backslash\mathbb{V}_+(\mathfrak{\widehat{d}})$. Therefore, we get
$$MGC(A)=\mathbb{V}_+(b_1,b_2,b_4,b_7,b_8,b_9,b_3^2-b_5b_6+b_3b_{10})\backslash\left(\mathbb{V}_+(a_4)\cup\mathbb{V}_+(b_{10})\cup\mathbb{V}_+ (b_3,b_5)\right).$$
\noindent
in $\mathbb{P}^{14}$. We can eliminate some of the variables and consider $MGC(A)$ to be the following variety:
$$MGC(A)=\mathbb{V}_+(b_3^2-b_5b_6+b_3b_{10})\backslash\left(\mathbb{V}_+(a_4)\cup\mathbb{V}_+(b_{10})\cup\mathbb{V}_+ (b_3,b_5)\right)\subset \mathbb{P}^8.$$
\noindent
Therefore, any minimal Gorenstein cover is of the form $G=R/\ann H$, where
$$H=a_1y_3^2+a_2y_2y_3+a_3y_1y_3+a_4y_2^2+a_5y_1y_2+b_3y_1y_3^2+b_5y_3^3+b_6y_1^2y_3-b_{10}y_1^4$$
\noindent
satisfies $b_3^2-b_5b_6+b_3b_{10}=0$, $a_4\neq 0$, $b_{10}\neq 0$ and either $b_3\neq 0$ or $b_5\neq 0$ (or both).

\noindent
Moreover, note that $\mathbb{V}_+(\mathfrak{c})\backslash\mathbb{V}_+(\mathfrak{a})=\mathbb{V}_+(\mathfrak{c_2})\backslash\mathbb{V}_+(\mathfrak{a})$, where $\mathfrak{c}=\mathfrak{c}_1\cap \mathfrak{c}_2\cap \mathfrak{c}_3$ is the primary decomposition of $\mathfrak{c}$ and $\mathfrak{c}_2=\mathfrak{b}_2+(v_2,v_1b_5-v_3b_3-v_3b_{10},v_1b_3-v_3b_6)$. Hence, any $K_H$ such that $K_H\circ H=I^\perp$ will be of the form $K_H=(L_1,L_2,L_3^2)$, where $L_1,L_2,L_3$ are independent linear forms in $R$ such that $L_3=v_1x_1+v_3x_3$, with $v_1b_5-v_3b_3-v_3b_{10}=v_1b_3-v_3b_6=0$.
\end{example}

\begin{example}
Consider $A=R/I$, with $R=\res[\![x_1,x_2,x_3]\!]$ and $I=(x_1^2,x_2^2,x_3^2,x_1x_2)$. Note that $\HF_A=\lbrace 1,3,2\rbrace$ and $\tau(A)=2$. Doing analogous computations to the previous examples, \emph{Singular} provides the following variety:
$$MGC(A)=\mathbb{P}^7\backslash\mathbb{V}_+(b_2^2-b_1b_3)$$
\noindent
The coordinates of points in $MGC(A)$ are of the form $(a_1:\dots:a_4:b_1:b_2:b_3:b_4)\in \mathbb{P}^7$ and they correspond to a polynomial
$$H=b_1y_1^2y_3+b_2y_1y_2y_3+b_3y_2^2y_3+b_4y_3^3+a_1y_3^2+a_2y_2^2+a_3y_1y_2+a_4y_1^2$$
\noindent
such that $b_2^2-b_1b_3\neq 0$. Any $G=R/\ann H$ is a minimal Gorenstein cover of colength 2 of $A$ and all such covers admit $(x_1,x_2,x_3^2)$ as the corresponding $K_H$.
\end{example}

\begin{example}
Consider $A=R/I$, with $R=\res[\![x_1,x_2,x_3,x_4]\!]$ and $I=(x_1^2,x_2^2,x_3^2,x_4^2,x_1x_2,x_1x_3,x_1x_4,x_2x_3,x_2x_4)$. Note that $\HF_A=\lbrace 1,4,1\rbrace$ and $\tau(A)=3$. Doing analogous computations to the previous examples, \emph{Singular} provides the following variety:
$$MGC(A)=\mathbb{V}_+(b_6b_{10}-b_9b_{16})\backslash\left(\mathbb{V}_+(d_1)\cup\mathbb{V}_+(d_2)\right)\subset \mathbb{P}^{12},$$
\noindent
where $d_1=(a_7a_9-a_8^2)$ and $d_2=(b_9^2b_{16}-b_{10}^3,b_6b_9-b_{10}^2,b_6^2-b_{10}b_{16})$. The coordinates of points in $MGC(A)$ are of the form $(a_1:\dots:a_9:b_6:b_9:b_{10}:b_{16})\in \mathbb{P}^{12}$ and they correspond to a polynomial
$$H=b_{16}y_3^3+b_6y_3^2y_4+b_{10}y_3y_4^2+b_9y_4^3+a_9y_1^2+a_8y_1y_2+a_7y_2^2+$$
$$+a_6y_1y_3+a_5y_2y_3+a_4y_3^2+a_3y_1y_4+a_2y_2y_4+a_1y_4^2$$
\noindent
such that $G=R/\ann H$ is a minimal Gorenstein cover of colength 2 of $A$. Moreover, any $K_H$ such that $K_H\circ H=I^\perp$ will be of the form $K_H=(L_1,L_2,L_3,L_4^2)$, where $L_1,L_2,L_3,L_4$ are independent linear forms in $R$ such that $L_4=v_3x_3+v_4x_4$, with $v_3b_9-v_4b_{10}=v_3b_6-v_4b_{16}=0$.
\end{example}

As in the case of colength 1, we now provide a table for the computation times of $MGC(A)$ of all isomorphism classes of local $\res$-algebras $A$ of length equal or less than 6 such that $\gcl(A)=2$.

{\small{
\begin{table}[h]
\begin{tabular}{|c|c|c|}
\hline
$\HF(R/I)$ & $I$ & t(s)\\
\hline
$1,3,1 $ &  $x_1x_2,x_1x_3,x_1^2,x_2^2,x_3^2$ & 0,42\\
$1,2,2,1 $ &  $x_1^2,x_1x_2^2,x_2^4$ & 0,18\\
$1,3,1,1 $ &  $x_1x_2,x_1x_3,x_2x_3,x_2^2,x_3^2-x_1^3$ & 3,56\\
$1,3,2 $ &  $x_1x_2,x_2x_3,x_3^2,x_2^2-x_1x_3,x_1^3$ & 4,4\\
&  $x_1x_2,x_3^2,x_1x_3-x_2x_3,x_1^2+x_2^2-x_1x_3$ & 1254,34\\
&  $x_1x_2,x_1x_3,x_2^2,x_3^2,x_1^3$ & 3,33\\
&  $x_1x_2,x_1x_3,x_2x_3,x_1^2+x_2^2-x_3^2$ & 4,61\\
&  $x_1^2,x_1x_2,x_2x_3,x_1x_3+x_2^2-x_3^2$ & 4,09\\
&  $x_1^2,x_1x_2,x_2^2,x_3^2$ & 0,45\\
$1,4,1 $ &  $x_1^2,x_2^2,x_3^2,x_4^2,x_1x_2,x_1x_3,x_1x_4,x_2x_3,x_2x_4$ & 242,28\\
\hline
\end{tabular}
\medskip
\caption[AlMGC2]{Computation times of $MGC(A)$ of local rings $A=R/I$ with $\ell(A)\leq 6$ and $\gcl(A)=2$.}
\label{table:TMGC2}
\end{table}}}
%%%%%%%%%%%%%%%%%%%%%%%%%%%%%%%%%%%%%%%%%%%%%%%%%%%%%%%%%%%%%%%%%%%%%%%%%%%%%%%%%%%%%%%%%
%%%%%%%%%%%%%%%%%%%%%%%%%%%%%%%%%%%%%%%%%%%%%%%%%%%%%%%%%%%%%%%%%%%%%%%%%%%%%%%%%%%%%%%%%
%%%%%%%%%%%%%%%%%%%%%%%%%%%%%%%%%%%%%%%%%%%%%%%%%%%%%%%%%%%%%%%%%%%%%%%%%%%%%%%%%%%%%%%%%
%%%%%%%%%%%%%%%%%%%%%%%%%%%%%%%%%%%%%%%%%%%%%%%%%%%%%%%%%%%%%%%%%%%%%%%%%%%%%%%%%%%%%%%%%
%%%%%%%%%%%%%%%%%%%%%%%%%%%%%%%%%%%%%%%%%%%%%%%%%%%%%%%%%%%%%%%%%%%%%%%%%%%%%%%%%%%%%%%%%
%%%%%%%%%%%%%%%%%%%%%%%%%%%%%%%%%%%%%%%%%%%%%%%%%%%%%%%%%%%%%%%%%%%%%%%%%%%%%%%%%%%%%%%%%
%%%%%%%%%%%%%%%%%%%%%%%%%%%%%%%%%%%%%%%%%%%%%%%%%%%%%%%%%%%%%%%%%%%%%%%%%%%%%%%%%%%%%%%%%
%%%%%%%%%%%%%%%%%%%%%%%%%%%%%%%%%%%%%%%%%%%%%%%%%%%%%%%%%%%%%%%%%%%%%%%%%%%%%%%%%%%%%%%%%

\end{document}